\documentclass[10pt,reqno,a4paper]{amsart}
\usepackage{bbm}
\usepackage{amsmath,amsthm,amssymb,amsfonts,bm,dsfont}
\usepackage{mathrsfs}
\usepackage[colorlinks, linkcolor=blue, citecolor=red]{hyperref}
\newtheorem{theorem}{Theorem}
\newtheorem{lemma}[theorem]{Lemma}

\newtheorem{proposition}[theorem]{Proposition}
\newtheorem{corollary}[theorem]{Corollary}

\newtheorem*{remarks*}{Remarks}
\newtheorem*{remark*}{Remark}
\numberwithin{equation}{section}
\numberwithin{theorem}{section}
\usepackage{graphicx}
\usepackage{xcolor}
\usepackage{academicons}
\usepackage{orcidlink}
\usepackage{enumitem,lipsum}

\newcommand{\pt}{\partial}

\newcommand{\T}{\mathbb{T}}
\newcommand{\R}{\mathbb{R}}
\newcommand{\N}{\mathbb{N}}
\newcommand{\C}{\mathbb{C}}
\newcommand{\Z}{\mathbb{Z}}
\newcommand{\be}{\begin{equation}}
\newcommand{\ee}{\end{equation}}
\newcommand{\D}{\mathbb{D}}
\newcommand{\weakto}{\rightharpoonup}
\newcommand{\ov}{\overline}
\newcommand{\eu}{\mathrm{e}}
\renewcommand{\O}{\mathcal{O}}

\newcommand{\eps}{\varepsilon}

\newcommand{\id}{\mathrm{Id}}

\newcommand{\LL}{\mathcal{L}}

\newcommand{\BB}{\mathsf{B}}
\newcommand{\Gc}{\mathcal{G}}
\newcommand{\Cc}{\mathcal{C}}
\renewcommand{\AA}{\mathsf{A}}
\newcommand{\dom}{\mathrm{Dom}}
\newcommand{\ran}{\mathrm{Ran}}
\renewcommand{\ker}{\mathrm{Ker}}
\newcommand{\Us}{\mathscr{U}}
\newcommand{\Ls}{\mathsf{L}}
\newcommand{\Bs}{\mathsf{B}}

\newcommand{\Ms}{\mathsf{M}}

\newcommand{\tr}{\mathrm{tr}}

\begin{document}

\title[Finite-time blow-up for CS-DNLS]{Finite-time blow-up for \\ the Calogero--Sutherland derivative NLS on $\T$}
\date{}
\author{Xi Chen}
\address{
Department of Mathematics and Computer Science,
University of Basel,
Spiegelgasse 1,
4051 Basel,
Switzerland
}
\email{xi01.chen@unibas.ch}

\author{Enno Lenzmann}
\address{
Department of Mathematics and Computer Science,
University of Basel,
Spiegelgasse 1,
4051 Basel,
Switzerland
}
\email{enno.lenzmann@unibas.ch}

\begin{abstract}
We study finite-time blow-up for the focusing Calogero--Sutherland derivative NLS on the torus with smooth initial data in the Hardy space $L^2_+(\T)$. For finite-gap potentials as initial data, the explicit solution formula reduces the dynamics to a finite-dimensional analytic family of contractions. This yields a complete description of every finite-time blow-up solution in this class: near blow-up time $T >0$, the solution decomposes into finitely many concentrating universal blow-up profiles and a smooth remainder, each bubble carries one unit of $L^2$-mass, the concentration points are distinct, and the blow-up rates are quantized according to
$$
\|u(t) \|_{H^s(\T)}
\sim_{s,u_0}
\frac{1}{(T-t)^{2s\nu}} \quad \mbox{as} \quad t \nearrow T \quad \mbox{for any $s > 0$}
$$
with some integer $\nu\geq1$.

Within an explicit subclass of finite-gap potentials, we identify an exact resonance condition that yields the existence of finite-time blow-up solutions. By a non-perturbative method, we construct single- and multi-bubble blow-up solutions with rate $\nu=1$ for every prescribed $L^2$-mass in $(1,\infty)\setminus \{ 2 \}$. In the complementary non-resonant regime, we prove global existence with uniform Sobolev bounds. The construction also shows instability of the blow-up solutions and non-chiral finite-time blow-up examples.

To the best of our knowledge, these results provide the first explicit, non-perturbative construction of finite-time blow-up for an integrable NLS-type equation on the torus, together with a complete classification of the singular dynamics within a natural finite-gap class.
\end{abstract}

\maketitle

\setcounter{tocdepth}{1}
\tableofcontents{}

\section{Introduction and main results}

\subsection{Setup and background}
We consider the {\em focusing Calogero--Sutherland derivative NLS} which can be written as
\begin{equation*} \tag{CS} \label{eq:CS}
\boxed{i \pt_t u = -\pt_x^2 u - 2 D \Pi(|u|^2) u \quad \mbox{with} \quad (t,x) \in \R \times \T}
\end{equation*}
Here $\T = \R/2 \pi \Z$ denotes the one-dimensional torus and we denote $D=-i \pt_x$. The operator $\Pi$ stands for the Cauchy--Szeg\H{o} projector, i.e.,
$$
\Pi \left ( \sum_{n \in \Z} \widehat{u}(n) \eu^{in x} \right ) = \sum_{n \geq 0} \widehat{u}(n) \eu^{inx} \quad \mbox{with} \quad \widehat{u}(n) = \frac{1}{2\pi} \int_0^{2 \pi} u(x) \eu^{-inx} \, dx,
$$
which is the orthogonal projection from $L^2(\T)$ onto the Hardy space
$$
L^2_+(\T) = \left \{ u \in L^2(\T) : \mbox{$\widehat{u}(n) = 0$ for $n <0$} \right \}.
$$
We recall the essential fact that elements in $L^2_+(\T)$ can be canonically identified via boundary values with the space of holomorphic functions $f : \D \to \C$ such that $f(z) = \sum_{n \geq 0} f_n z^n$ with $\sum_{n \geq 0} |f_n|^2< \infty$. In our analysis, we {will make strong use} of this classical fact; see \cite{GaMaRo} for a general background on Hardy spaces. Furthermore, we will be mainly interested in so-called {\em chiral solutions} of \eqref{eq:CS}, i.e., {solutions whose initial data belong} to the Hardy--Sobolev spaces
$$
H^s_+(\T) := L^2_+(\T) \cap H^s(\T) \quad \mbox{and} \quad H^\infty_+(\T) := \bigcap_{s \geq 0} H^s_+(\T),
$$  
where \eqref{eq:CS} features an explicit formula for its solutions. However, we shall also state some results for the non-chiral case, where {membership in} $L^2_+(\T)$ will be dropped; see below.

\medskip
Let us mention  that \eqref{eq:CS} formally arises from considering the continuum limit of discrete Calogero--Moser systems, which form an intriguing class of classically completely integrable systems; see \cite{AbBeWi} for a formal derivation of \eqref{eq:CS} in the physics literature as well as \cite{Ma-23} for a formal analysis of so-called multiphase solutions. With regard to physical applications, let us also mention that the  defocusing variant of \eqref{eq:CS}, where the sign `$-$' in front of {the} nonlinearity is changed to `$+$', {was} previously derived in \cite{PeGr-95} under the name `intermediate nonlinear Schr\"odinger equation' as an effective model in fluid dynamics. 

\medskip
For the mathematical study of \eqref{eq:CS}, the first rigorous results were obtained by Badreddine in \cite{Ba-24,Ba-25}. More
precisely, she established the following facts.
\begin{itemize}
\item {\em Lax pair structure} on $L^2_+(\T)$ with the self-adjoint Lax operator 
\be \label{def:Lax}
L_{u_0} = D - T_{u_0} T_{\bar{u}_0}
\ee 
with  the Toeplitz operator $T_g f = \Pi (gf)$ with symbol $g \in L^\infty(\T)$.
\item {\em Global well-posedness} in $H^s_+(\T)$ for any $s \geq 0$ for initial data $u_0$ below the critical $L^2$-mass threshold
\be \label{ineq:small_mass}
 \|u_0\|_{L^2}<1.
 \ee
\item Derivation of an \emph{explicit formula} for solutions of \eqref{eq:CS}; see \eqref{eq:EF} below.
\item Complete classification of {\em finite-gap potentials}; see Subsection \ref{subsec:finite_gap} {below}.
\end{itemize}
{{Condition \eqref{ineq:small_mass}, together with the explicit formula, plays a pivotal role}} in \cite{Ba-24}. We also mention the further recent work \cite{AlBrChDo-24, ABLa-26}, where the explicit formula for \eqref{eq:CS} is used to design a new numerical scheme.

\subsubsection*{Comparison to the real line problem} When $\T$ is replaced by $\R$, we remark that \eqref{eq:CS} is usually referred to as the focusing
{\em Calogero--Moser derivative NLS} given by
\begin{equation*} \label{eq:CM} \tag{CM}
i \pt_t u = -\pt_x^2 u - 2 D \Pi_{\R} (|u|^2) u \quad \mbox{with} \quad (t,x) \in \R \times \R,
\end{equation*}
with the corresponding Cauchy--Szeg\H{o} projector on the real line, i.e., in Fourier space we have
$$
\widehat{(\Pi_\R {f})}(\xi) = \mathbf{1}_{\{ \xi \geq 0 \}} \widehat{f}(\xi) \quad \mbox{for $f \in L^2(\R)$}.
$$
Not surprisingly, it turns out that \eqref{eq:CM} also features a Lax pair structure on the corresponding Hardy space 
$$
L^2_+(\R) = \{ f \in L^2(\R) : \mbox{$\widehat{f}(\xi) = 0$ for a.e.~$\xi < 0$} \}.
$$ 
In fact, the body of mathematical results on \eqref{eq:CM} is much more advanced when compared to \eqref{eq:CS}. For the reader's convenience, we give a brief summary as follows. The rigorous mathematical analysis of \eqref{eq:CM} was initiated in \cite{GeLe-24}, where  global well-posedness in $H^s_+(\R)=L^2_+(\R) \cap H^s(\R)$ with $s \geq 1$ was shown when $\|u_0\|_{L^2}^2 \leq 2 \pi$, together with the explicit construction of weakly turbulent $N$-solitons blowing up at infinite time. Later on, global well-posedness in the scaling-critical space $L^2_+(\R)$ with $\| u_0 \|_{L^2}^2 < 2 \pi$ was obtained in \cite{KiLaVi-25} by means of the explicit formula and equicontinuity arguments. On the other hand, smooth finite-time blow-up solutions for \eqref{eq:CM} were first constructed in \cite{KiKiKw-24} by a delicate perturbative analysis for smooth initial data in $L^2_+(\R)$ (and beyond Hardy space setting) with slightly supercritical $L^2$-mass; see also \cite{HoKo-24} for another blow-up result. Also, we mention the result in \cite{KiKw-24} on soliton resolution for \eqref{eq:CM}.  Further blow-up results for so-called non-chiral solutions, i.e., when the Hardy space assumption on the initial data is dropped, were obtained in \cite{JeKi-24, JeKiKiKw-26, KiKiKw-24}. More recently, an inverse scattering transform was proposed for \eqref{eq:CM} in \cite{Frank-Read-2025}. Finally, let us also mention the very recent work \cite{KiMaVi-26}, where the Hamiltonian structures for \eqref{eq:CS} and \eqref{eq:CM} are studied.

\medskip
Existing approaches to finite-time blow-up for \eqref{eq:CM} on $\R$ do not appear to adapt to the periodic problem. Our approach is instead based directly on the explicit formula for \eqref{eq:CS}, which turns the formation of singularities for finite-gap data into a finite-dimensional spectral problem. We describe this mechanism and its main consequences in detail below. In addition, let us mention that the use of explicit formulae for integrable PDEs (initiated by P.~G\'erard in \cite{Ge-23}) with Lax pairs on Hardy spaces (e.g.~the Benjamin--Ono equation, the cubic Szeg\H{o} equation, CS, CM, and the half-wave maps equation) {has} turned out to be extremely powerful, yielding results on scaling-critical global well-posedness, weak turbulence, zero-dispersion limits, classification of solitary waves, and soliton resolution; see \cite{Ba-24, Bad-24b, Ba-25, BKV-26, Chen-25, Chen-26, Ga-23, GaGe-26, GaGeMi-26, Ge-23, Ge-24, GeGr-15, GeHe-26, GeLe-24, GeLe-25, GeLe-26, GePu-23, GePu-24, HoKo-24, KiLaVi-25, Ma-26}. The present paper can be seen as a novel addition to this list by proving and analyzing finite-time blow-up via explicit formulae.  In our companion work \cite{ChLe-26}, we implement this strategy of constructing finite-time blow-up via an explicit formula for \eqref{eq:CM} posed on the non-compact domain $\R$.

\subsection{Finite-gap potentials}
\label{subsec:finite_gap}
As initial data $u_0$ for \eqref{eq:CS}, we consider the class of so-called {\em finite-gap potentials}, whose study was initiated in \cite{Ba-25}. Let us recall the defining property of these potentials. For $u_0 \in L^2_+(\T)$, the corresponding Lax operator 
$$
L_{u_0}=D - T_{u_0} T_{\bar{u}_0},
$$ 
can be defined via its quadratic form and it is found to be a self-adjoint operator on $L^2_+(\T)$ with a suitable domain $\dom (L_{u_0})$. {Due to compactness of its resolvent}, the spectrum $\sigma(L_{u_0})$ is discrete with eigenvalues
$$
\lambda_0(u_0) \leq \lambda_1(u_0) \leq \ldots \leq \lambda_{n}(u_0) \leq \lambda_{n+1}(u_0) \leq \ldots \to +\infty \quad \mbox{as $n \to \infty$}.
$$ 
We say that $u_0 \in L^2_+(\T)$ is a {\em finite-gap potential} if for some integer $n_0 \geq 1$, it holds
$$
\lambda_{n}(u_0) - \lambda_{n-1}(u_0) = 1 \quad \mbox{for $n \geq n_0$},
$$
i.e., the difference between two consecutive eigenvalues of $L_{u_0}$ is equal to 1, except for finitely many cases. From \cite{Ba-25} we recall the following result: The finite-gap potentials $u_0 \in L^2_+(\T)$ for \eqref{eq:CS} belong to one of the following categories.
\begin{itemize}
\item {\em Trivial traveling wave profiles} with $u_0(z) = C z^{N}$ with $C\in \C$ and integer $N \in \Z_{\geq 0}$.
\item {\em Rational functions} of the form\footnote{{Our convention for $p_1, \ldots, p_r$ differs from \cite{Ba-25} by complex conjugation.} In \cite{Ba-25}, the finite-gap potentials for the defocusing version of \eqref{eq:CS} are also classified, but we do not need this here.}
\be \label{eq:finite_gap}
\boxed{u_0(z) = z^{m_0} \prod_{j=1}^r \left ( \frac{z-p_j}{1-\bar{p}_j z} \right )^{m_j-1} \left ( a + \sum_{j=1}^r \frac{c_j}{1-\bar{p}_j z} \right )}
\ee
where $p_1, \ldots, p_r \in \D \setminus \{0 \}$ with $p_j \neq p_k$ for $j \neq k$ and, for some $N \in \N$, $0 \leq m_0 \leq N-1$, $1 \leq m_1, \ldots, m_r \leq N$ such that
$m_0 + \sum_{j=1}^r m_j = N$, and the parameters $a, c_1, \ldots, c_r \in \C$ satisfy the constraints
\be \label{eq:constraints}
{\bar{a} c_j + \sum_{k=1}^r \frac{c_j \bar{c}_k}{1- \bar{p}_j p_k} = m_j}
\ee
for $j=1, \ldots, r$ and $a \neq 0$ if $m_0 \neq 0$.
\end{itemize}
In particular, we see that any finite-gap potential $u_0 \in L^2_+(\T)$ is a rational function with poles outside the closed unit disk $\ov{\D}$ and hence $u_0$ is smooth on $\T$. From now on, we assume that $u_0 \in H^\infty_+(\T)$ without loss of generality. We remark that the class of finite-gap potentials is not limited to small $L^2$-mass, but that in fact any $L^2$-mass $0 < \| u_0 \|_{L^2}^2 < \infty$ can be reached within this class. Therefore, the class of finite-gap potentials appears as a natural testing ground for studying possible blow-up solutions of \eqref{eq:CS}, as we will see below.

\medskip
Suppose that $u \in C(I; H^\infty_+(\T))$ is a smooth solution of \eqref{eq:CS} with initial datum $u(0) = u_0 \in H^\infty_+(\T)$, which is not necessarily a finite-gap potential. In \cite{Ba-24}, it was shown that we have the following {\em explicit formula} given by
\begin{equation*} \tag{EF} \label{eq:EF}
\boxed{u(t,z) = \left \langle (\id - z {\eu^{-it}} \eu^{-2it L_{u_0}} S^*)^{-1} u_0, 1 \right \rangle \quad \mbox{for} \quad (t,z) \in I \times \D}{.}
\end{equation*}
Here $\langle f, g\rangle = \frac{1}{2 \pi} \int_{\T} f \ov{g}$ denotes the inner product on $L^2_+(\T)$, whereas $S^*$ is the {\em backward shift} (or {\em left shift}) acting on $L^2_+(\T)$ given by
$$
(S^* f)(z) = S^* \left ( \sum_{n=0}^\infty f_n z^n \right ) := \sum_{n=0}^\infty f_{n+1} z^n.
$$
For finite-gap potentials $u_0 \in H^\infty_+(\T)$, it turns out that the analysis of \eqref{eq:EF} reduces to a {\em finite-dimensional} subspace $W$, which is invariant under both $L_{u_0}$ and $S^*$, and contains the vectors $u_0$ and $1$. As a consequence, in the setting of finite-gap potentials, we obtain that \eqref{eq:EF} reduces to $W$ so that
\be
u(t,z) = \left \langle (\id - z A_t)^{-1} u_0, 1 \right \rangle \quad \mbox{for} \quad (t,z) {\in} I \times \D
\ee
with $t$-dependent family of contractions
$$
A_t := \eu^{-it \LL} S^*|_W \quad \mbox{with} \quad \LL := 2 L_{u_0} + \id
$$
acting on the invariant finite-dimensional subspace $W$. This fact will be pivotal for proving the following main results below:
\begin{enumerate}
\item Universal blow-up dynamics and quantized blow-up rates for finite-gap potentials $u_0 \in H^\infty_+(\T)$ as initial data.
\item Explicit construction of finite-time blow-up solutions for a suitable class of finite-gap potentials $u_0 \in H^\infty_+(\T)$ in \eqref{eq:finite_gap}.
\end{enumerate}
In the next subsections, we will state the main theorems which yield (1) and (2).

\subsection{Universality of blow-up for finite-gap potentials}

As our first main result, we obtain the following complete description of finite-time blow-up for \eqref{eq:CS} with finite-gap potentials as initial data.

\begin{theorem}[Universal blow-up dynamics and quantization of blow-up rates] \label{thm:universal}
Let $u_0 \in H^\infty_+(\T)$ be a finite-gap potential. Suppose  the solution $u \in C([0,T); H^\infty_+(\T))$ of \eqref{eq:CS} with $u(0) = u_0$ blows up at time $0 < T < \infty$. Then there exists an integer $1 \leq M < \| u_0 \|_{L^2}^2$ such that
$$
u(t,z) - \sum_{j=1}^M \frac{\beta_j(t)}{1- \omega_j(t)z} \rightarrow u_*(z) \quad \mbox{in} \quad H^\infty_+(\T) \quad \mbox{as} \quad t \nearrow T,
$$
where $\beta_1(t), \ldots, \beta_M(t)$ and $\omega_1(t), \ldots, \omega_M(t)$ are analytic functions defined for $t \in (T-\eps_0, T+ \eps_0)$ with some $\eps_0 > 0$ small, and we have
$$
{| \beta_j(t)|^2 = 1 - |\omega_j(t)|^2 + \O((T-t)^{2\nu_j+1}),}
$$
$$
|\omega_j(t)|^2 = 1-c_j (T-t)^{2\nu_j} + \O((T-t)^{2\nu_j + 1}),
$$
with some constants $c_j > 0$ and some integers $\nu_1, \ldots, \nu_M \geq 1$.

 Moreover, the following properties hold.

\begin{enumerate}
\item[(i)] \textbf{Quantization of blow-up rate and mass:} For any $s > 0$, we have
$$
\| u(t) \|_{H^s(\T)} \sim_{s,u_0} \frac{1}{(T-t)^{2s \nu}} \quad \mbox{as} \quad t \nearrow T
$$ 
with the integer $\nu = \max_{1 \leq j \leq M} \nu_j$. Moreover, it holds that
$$
\| u_* \| _{L^2}^2  =  \| u_0 \|_{L^2}^2 - M.
$$

\item[(ii)] \textbf{No bubble tower:} The limits $\omega_{j,*} = \lim_{t \nearrow T} \omega_j(t) \in \pt \D$ exist for $1 \leq j \leq M$ and satisfy $\omega_{j,*} \neq \omega_{k, *}$ for $k \neq j$.

\item[(iii)] \textbf{Non-triviality of remainder:} We have $u_* \neq 0$.
\end{enumerate}
\end{theorem}

\begin{remarks*}
{\em 1) With obvious changes, the corresponding result holds true for blow-up of $u(t)$ at finite negative time. Also, we note the interesting fact {that} the functions $\beta_j(t)$ and $\omega_j(t)$ are analytic for $t$ close to $T$; in particular, these functions are well-defined past the blow-up time.

2) Note the strict inequality $M < \| u_0 \|_{L^2}^2$ for the integer $M \geq 1$. In particular, this shows that no finite-time blow-up solution for \eqref{eq:CS} with minimal mass $\|u_0 \|_{L^2}^2 = 1$ can exist, within in the class of finite-gap potentials. In our companion work \cite{ChLe-26}, we largely extend this result beyond the setting of finite-gap potentials, by ruling out blow-up in finite and infinite time for solutions $u \in C(I; H^s_+(\T))$ of \eqref{eq:CS} for any $s > 3/2$ with minimal mass $\| u_0 \|_{L^2}^2 = 1$.

3) The parameters $\omega_1(t), \ldots, \omega_M(t)$ defined for $t$ close to $T$ correspond to the eigenvalues of the contraction $\eu^{-it \LL} S^*|_W$, where the limits $\omega_{j,*} \in \pt \D$ are the simple unimodular eigenvalues of $\eu^{-iT \LL} S^* |_W$ at blow-up time $T$. A key fact proven below is that these unimodular eigenvalues must be (algebraically) simple, and hence standard perturbation theory for simple eigenvalues for analytic families of matrices can be applied. The blow-up rate of $\| u(t) \|_{H^s}$ is then determined by the order at which the perturbed eigenvalues $\omega_1(t), \ldots, \omega_M(t)$ approach the unit circle $\pt \D$ as $t \nearrow T$. By analyticity, {these rates} must be given by some integers $\nu_1, \ldots, \nu_M \geq 1$ and hence $\nu = \max_{1 \leq j \leq M} \nu_j$ determines the blow-up rate.

{4) The construction of finite-gap blow-up solutions with higher integer rates $\nu\geq2$ is not pursued in the present paper.}

5) Note also the recent result in \cite{JeKiKiKw-26}, which proves a dichotomy {between quantized and `exotic' blow-up rates} for \eqref{eq:CM} on $\R$ by using a completely different strategy. It is an interesting open question whether the analogous result (or some sharpened version) can be shown in the torus setting beyond the case of finite-gap potentials.
}
\end{remarks*}

\subsection{Existence of finite-time blow-up}
We now turn to the existence of finite-time blow-up for finite-gap potentials. To give an affirmative answer to this question, we consider the following special case of \eqref{eq:finite_gap} with
\be
\label{eq:finite_gap_2}
u_0(z) = z^{m} \beta_p(z) \left (a + \frac{c}{1- \bar{p} z} \right ) \quad \mbox{with} \quad \beta_p(z) := \frac{z-p}{1- \bar{p} z}
\ee
with some integer $m \geq 0$, $p \in \D \setminus \{0 \}$ and $a,c \in \C$ satisfying the constraint
\be \label{eq:constraints2}
\bar{a} c+ \frac{|c|^2}{1-|p|^2} = 2
\ee
with $a\neq 0$ if $m \neq 0$. {{Notice that this corresponds to \eqref{eq:finite_gap} with $m_0=m$, $r=1$, $m_1 =2$, and integer $N=m + 2 \geq 2$.}} To further locate  blow-up data in this subclass, the analysis below will show that we need to additionally impose the `resonance' condition given by
\be \label{eq:resonance}
2a + c = 0.
\ee
Assuming this extra condition, a straightforward calculation yields
\be \label{eq:ac_blowup}
 a = \eu^{i \phi} \sqrt{\frac{1-|p|^2}{1 + |p|^2}} \quad \mbox{and} \quad c = -2 \eu^{i \phi} \sqrt{\frac{1-|p|^2}{1+|p|^2}}
\ee
for some $\phi \in \T$. Thus we are finally led to finite-gap potentials of the form
\be \label{def:u0_blowup}
\boxed{u_0(z) = \eu^{i \phi} \sqrt{\frac{1-|p|^2}{1+|p|^2}} z^{m} \beta_p(z) \left ( 1 - \frac{2}{1-\bar{p}z} \right )}
\ee 
where $\phi \in \T$, $m \in \mathbb{Z}_{\ge 0}$, and $p \in \D \setminus \{ 0 \}$ can be chosen arbitrarily. For the $L^2$-mass of the blow-up initial data $u_0$ in \eqref{def:u0_blowup}, it is {easily checked} that
 \be \label{ineq:L2_mass_single}
    1 < \| u_{0} \|_{L^2}^2 = \frac{1+3|p|^2}{1+|p|^2} < 2
 \ee
as $|p|$ ranges over $(0,1)$. Thus the $L^2$-mass of the family of initial data as above covers the entire interval $(1,2)$.

\begin{theorem}[Existence of finite-time blow-up, single-bubble case]
\label{thm:blowup}Let \(u_0 \in H^\infty_+(\T)\) be a finite-gap potential given by \eqref{def:u0_blowup} with {some $\phi\in\T$, $m\in\Z_{\geq0}$, and $p\in\D\setminus\{0\}$}. Then the corresponding solution $u \in C([0,T); H^\infty_+(\T))$ of \eqref{eq:CS} with $u(0)=u_0$ blows up at finite time
$$
{T = \frac{\pi(1-|p|^2)}{4|p|} > 0.}
$$

Moreover, the conclusions of Theorem \ref{thm:universal} hold  with 
$$
M=1 \quad \mbox{and} \quad \nu = 1.
$$
Consequently, we have
$$
u(t,z) - \frac{\beta(t)}{1-\omega(t)z} \to u_*(z) \quad \mbox{in} \quad  H^\infty_+(\T) \quad \mbox{as} \quad t \nearrow T,
$$
where the functions $\beta(t)$ and $\omega(t)$ defined for $t$ close to $T$ satisfy
$$
|\beta(t)|^2 = 1- |\omega(t)|^2 + \O((T-t)^3), \quad |\omega(t)|^2 = 1- c (T-t)^2 + \O((T-t)^{3})
$$
with some constant $c > 0$. Furthermore, for any $s > 0$, it holds
$$
\| u(t) \|_{H^s} \sim_{s,u_0} \frac{1}{(T-t)^{2s}} \quad \mbox{as} \quad t \nearrow T,
$$
and we have $\| u_* \|_{L^2}^2 = \| u_0 \|_{L^2}^2 - 1$ and $u_* \neq 0$.
\end{theorem}

\begin{remark*}{\em 
  In \cite{ChLe-26b}, we construct an explicit family of smooth rational data that lead to finite-time blow-up solutions of \eqref{eq:CM} on $\R$ by using the corresponding explicit formula on the real line. The approach in \cite{ChLe-26b} is again non-perturbative and extends the recent results in \cite{KiKiKw-24}.
 } 
 \end{remark*}
 
 Next, we prove existence of {multi-bubble blow-up solutions} for any given integer $M \geq 2$ in Theorem \ref{thm:universal}. Here, we observe that the single-bubble blow-up solutions obtained in Theorem \ref{thm:blowup} can be transformed into multi-bubble blow-up solutions by means of the {\em covering symmetry} given by
 \be \label{eq:cov}
 \boxed{\Cc_k u(t,z) := k^{1/2} u(k^2 t, z^k) }
 \ee
 for any fixed integer $k \geq 1$, which is seen to map solutions of \eqref{eq:CS} into solutions. {{See Appendix \ref{sec:app} for details; we include the trivial case $k=1$ for notational convenience.}} Note, however, that the covering transform $\Cc_k$ does not preserve the $L^2$-mass for $k \geq 2$, but rather we have the relation
 \be
 \| \Cc_k u(t) \|_{L^2}^2 = k \| u(t) \|_{L^2}^2 = k \| u_0 \|_{L^2}^2{.}
 \ee
 Based on this covering symmetry for \eqref{eq:CS}, we can derive explicit multi-bubble blow-up solutions emanating from the single-bubble blow-up solutions found in Theorem \ref{thm:blowup}. Here is the precise statement.
 
 \begin{theorem}[Existence of finite-time blow-up, multi-bubble case] \label{thm:blowup_multi}
Suppose that $k \geq 1$ is an integer and let 
$$
u_0^{(k)}(z) := k^{1/2} u_0(z^k) \quad \mbox{for $z \in \D$},
$$
where $u_0 \in H^\infty_+(\T)$ is a finite-gap potential as in Theorem \ref{thm:blowup} above. Then $u_0^{(k)} \in H^\infty_+(\T)$ is a finite-gap potential and the corresponding solution $u^{(k)} \in C([0,T_k); H^\infty_+(\T))$ of \eqref{eq:CS} is given by \eqref{eq:cov} and blows up at finite time 
$$
T_k = \frac{1}{k^2} \cdot \frac{ \pi(1-|p|^2)}{4|p|} > 0. 
$$
Moreover, the conclusions of Theorem \ref{thm:universal} hold with
$$
M=k \quad \mbox{and} \quad \nu_1 = \ldots = \nu_M = 1.
$$
In addition, the limits $\omega_{1,*}, \ldots, {\omega_{M,*}} \in \pt \D$ form a regular $k$-gon given by
$$
\omega_{j,*} = \omega_* \eu^{2\pi i (j-1)/k} \quad \mbox{for} \quad j=1, \ldots, M,
$$
with some $\omega_* \in \pt \D$.
 \end{theorem}
 
 \begin{remarks*}
 {\em 1) Recall that the range of possible $L^2$-masses of $u_0 \in H^\infty_+(\T)$ in Theorem \ref{thm:blowup} covers the whole interval $(1,2)$. Therefore, by applying the theorem above, we can construct explicit finite-time blow-up solutions for \eqref{eq:CS} in the {range}
  $$
 {\| u_0^{(k)} \|_{L^2}^2 \in (1,\infty)\setminus\{2\}}
 $$
 {by choosing $k \geq 1$ and the parameter $0 < |p| < 1$ appropriately. Indeed, the mass range for fixed $k$ is $(k,2k)$, and the union of these intervals over all $k\geq1$ equals $(1,\infty)\setminus\{2\}$.} 
 
 {2) The covering symmetry preserves the rate $\nu=1$ of the single-bubble family. Constructions with higher integer blow-up rates are not pursued here.}
 }
 \end{remarks*}
 
 \subsection{Instability of blow-up solutions}
 
 It is a natural question whether the blow-up solutions in Theorem \ref{thm:blowup_multi} are stable under small perturbations of the initial datum -- at least, within the class of finite-gap potentials. In fact, we can show that the blow-up solutions are not stable, by considering finite-gap potentials $u_0 \in H^\infty_+(\T)$ as in \eqref{eq:finite_gap_2} but assuming that the resonance condition \eqref{eq:resonance} fails. As a starting point, we have the following result.

\begin{theorem} \label{thm:gwp_finite_gap}
Let $u_0 \in H^\infty_+(\T)$ be a finite-gap potential of the form \eqref{eq:finite_gap_2} with some integer $m \geq 0$, $p \in \D \setminus \{ 0 \}$ and $a,c \in \C$ satisfying {\eqref{eq:constraints2}}. Assume further that
$$
2a + c \neq 0.
$$
Then the corresponding solution $u \in C(\R; H^\infty_+(\T))$ of \eqref{eq:CS} with $u(0)=u_0$ exists globally in time with a priori bounds
$$
\sup_{t \in \R} \| u(t) \|_{H^s} \lesssim_{u_0,s} 1 \quad \mbox{for any $s > 0$}.
$$
\end{theorem}

{{By first approximating the one-pole datum in Theorem \ref{thm:blowup} by non-resonant one-pole finite-gap potentials and then applying the same covering transformation, we obtain the following instability result.}}

\begin{corollary}[Instability of blow-up solutions] \label{coro:instability}
For every $u_0^{(k)} \in H^\infty_+(\T)$ as in Theorem \ref{thm:blowup_multi}, there exists a sequence of finite-gap potentials $u_{0,n} \in H^\infty_+(\T)$  such that  $u_{0,n} \to u_0^{(k)}$ in $H^\infty_+(\T)$ as $n \to \infty$, where the corresponding solutions $u_n \in C(\R; H^\infty_+(\T))$ of \eqref{eq:CS} with $u_{n}(0) = u_{0,n}$ {exist} globally in time with a priori bounds on all $H^s$-norms.
\end{corollary} 

{
\begin{corollary}[Instability of plane waves] \label{coro:plane_wave_instability}
For every pair of integers $\nu,k\geq1$ and every $\theta\in\T$, the plane wave
\[
u_{\rm pw}(t,z)
=
\eu^{i\theta}\sqrt{k}\,
\eu^{-i(\nu k)^2t}z^{\nu k}
\]
can be approximated in $H^\infty_+(\T)$ at $t=0$ by finite-gap initial data whose corresponding solutions exhibit finite-time $k$-bubble blow-up.
\end{corollary}

\begin{proof}
Choose $p_n\to0$ in Theorem \ref{thm:blowup}, take $m=\nu-1$, and choose the phase $\phi=\theta+\pi$. The resulting initial data converge in $H^\infty_+(\T)$ to $\eu^{i\theta}z^\nu$. Applying $\Cc_k$ gives finite-gap initial data converging to $\eu^{i\theta}\sqrt{k}z^{\nu k}$, while Theorem \ref{thm:blowup_multi} shows that every member of the approximating sequence produces finite-time $k$-bubble blow-up.
\end{proof}

\begin{corollary}[Non-chiral blow-up solutions]
Let $u$ be a single-bubble blow-up solution from Theorem \ref{thm:blowup} with $m=0$. For every integer $\ell\geq1$, the Galilean transform $\Gc_{-\ell}u$ is a non-chiral smooth solution of \eqref{eq:CS} that blows up at the same time and with the same Sobolev blow-up rate as $u$.
\end{corollary}

\begin{proof}
By Lemma \ref{lem:Galilean}, $\Gc_{-\ell}u$ solves \eqref{eq:CS}, and its Sobolev norms are equivalent to those of $u$. Moreover,
\[
\widehat{u_0}(0)
=\eu^{i\phi}\sqrt{\frac{1-|p|^2}{1+|p|^2}}\,p\neq0,
\]
so the initial datum of $\Gc_{-\ell}u$ has a nonzero Fourier coefficient at frequency $-\ell$. Hence it is not chiral.
\end{proof}
}

\subsection{Strategy of the proofs and main novelties}

The starting point of our analysis is the explicit formula \eqref{eq:EF}. 
For a finite-gap potential $u_0 \in H^\infty_+(\T)$, the associated finite Blaschke product
$\psi_{u_0}$ determines a finite-dimensional space
\[
W=(z\psi_{u_0}L^2_+(\T))^\perp
\]
which is invariant under both $L_{u_0}$ and $S^*$ and contains $u_0$ and
$1$. Consequently, the explicit formula reduces to
\[
u(t,z)
=
\left\langle
(\id-zA_t)^{-1}u_0,1
\right\rangle,
\qquad
A_t=\eu^{-it\LL}S^*|_W ,
\]
where $\LL=2L_{u_0}+\id$ and $(A_t)_{t\in\R}$ is an analytic family of
finite-dimensional contractions. Combined with the stability principle in Section \ref{sec:lwp},
this gives a purely spectral characterization of blow-up: finite-time
blow-up occurs precisely when an eigenvalue of $A_t$ reaches the unit
circle. Thus the nonlinear singularity formation problem is reduced to
the spectral dynamics of a finite-dimensional, generally non-normal,
family of contractions.

The main step in the universal blow-up analysis is to understand the
unitary part of $A_T$ at a blow-up time $T$. A virial-type identity shows
that every normalized eigenvector $f$ associated with a unimodular
eigenvalue satisfies
\[
|\langle u_0,f\rangle|=1.
\]
This identity has several consequences: the unimodular eigenvalues are
simple and mutually distinct, and their number is controlled by the
$L^2$-mass. Simplicity allows us to continue these eigenvalues
analytically as branches $\omega_j(t)$ for $t$ close to $T$. Decomposing
the resolvent according to the corresponding Riesz projections then
gives
\[
u(t,z)
=
\sum_{j=1}^M
\frac{\beta_j(t)}{1-\omega_j(t)z}
+
w(t,z),
\]
where $w(t,\cdot)$ remains holomorphic in a fixed neighborhood of
$\overline{\D}$. The contraction structure moreover yields
\[
|\beta_j(t)|^2
\sim
1-|\omega_j(t)|^2.
\]
Since $1-|\omega_j(t)|^2$ is a non-negative real-analytic function
vanishing at $t=T$, its first nonzero term has even order:
\[
1-|\omega_j(t)|^2
=
c_j(T-t)^{2\nu_j}
+O((T-t)^{2\nu_j+1}),
\qquad
\nu_j\in\N.
\]
This is the mechanism behind both the quantization of the Sobolev
blow-up rates and the quantization of the concentrating $L^2$-mass.
The simplicity and separation of the limiting unimodular eigenvalues
also exclude bubble towers.

For the existence of blow-up solutions, we specialize to the one-pole finite-gap
family \eqref{eq:finite_gap_2}. After exploiting the symmetries of the
equation, the relevant spectral calculation reduces to an explicit
low-dimensional matrix problem. We show that the family $A_t$ reaches
the unit circle if and only if the algebraic resonance condition
\[
2a+c=0
\]
holds. In the resonant case the first contact with the unit circle occurs
at the explicit time
\[
T=\frac{\pi(1-|p|^2)}{4|p|},
\]
and a second-order expansion of the corresponding eigenvalue gives
$\nu=1$. In the non-resonant case, by contrast, the spectrum of $A_t$
stays uniformly inside the unit disk for all $t$, which yields global
existence together with uniform bounds in every Sobolev space. The
covering symmetry then turns the single-bubble solutions into
multi-bubble solutions, while the Galilean symmetry produces further
examples outside the chiral Hardy-space setting.

The resulting picture is therefore governed by a sharp
resonance/non-resonance dichotomy. In contrast with perturbative
constructions of blow-up for the corresponding real-line problem, the
blow-up mechanism here is read directly from the explicit formula:
singularity formation is encoded by eigenvalues of a finite-dimensional
contraction reaching the unit circle. Besides giving, to our knowledge,
the first finite-time blow-up solutions for the focusing
Calogero--Sutherland derivative NLS on $\T$, this approach yields a
classification of all finite-time blow-up dynamics within the
finite-gap class, including quantization of the blow-up mass and rate,
as well as explicit multi-bubble solutions and their instability.

In our companion work \cite{ChLe-26}, we adapt the strategy of using the explicit formula to study finite-time blow-up for \eqref{eq:CS} on $\R$.

\subsection*{Acknowledgements}

The authors are grateful to Rana Badreddine for valuable discussions and exchanges at the Mathematisches Forschungsinstitut Oberwolfach in March 2026. They also thank Patrick G\'erard for
helpful suggestions related to this work and for the kind hospitality of the {Institut} Henri Poincar\'e in June and July 2026. Financial support by the Swiss National Science Foundation (SNSF) under Grant No.~204121 is gratefully acknowledged. The authors also acknowledge the use of AI tools  in the initial phase of this project. All mathematical arguments and proofs in the final manuscript were written and checked by the authors.

\section{Local well-posedness, explicit formula, and stability principle}

\label{sec:lwp}

 From \cite{Ba-24}, we recall that \eqref{eq:CS} is locally well-posed in $H^s_+(\T)$ for any $s > \frac 3 2$. Furthermore, by persistence of regularity, we obtain the following result for smooth initial data in $H^\infty_+(\T)$. Since we ultimately deal with finite-gap potentials, we restrict our discussion to smooth solutions henceforth. Also, we recall the explicit formula for smooth solutions of \eqref{eq:CS} derived in \cite{Ba-24}. We record these facts as follows.

\begin{proposition}[LWP in $H^\infty_+(\T)$ and explicit formula] \label{prop:lwp}
For any $u_0 \in H^\infty_+(\T)$, there exists a unique solution $u \in C(I; H^{\infty}_+(\T))$ of \eqref{eq:CS} with $u(0)=u_0$, where $I \subset \R$ with $0 \in I$ denotes its maximal time interval of existence. The following properties hold.
\begin{enumerate}
\item[(i)] \textbf{Infinite set of conserved quantities}: For $t \in I$, the quantities
$$
I_k[u(t)] = \langle L_{u(t)}^k u(t), u(t) \rangle \quad \mbox{for} \quad k=0,1,2, \ldots
$$ 
{are conserved, where the Lax operator is $L_u = D - T_u T_{\bar{u}}$.}
\item[(ii)] \textbf{Explicit formula:} We have
$$
u(t,z) = \left \langle (\id -z \eu^{-it \LL} S^*)^{-1} u_0,1 \right \rangle \quad \mbox{for $(t,z) \in I \times \D$} 
$$
{where} $\LL = 2 L_{u_0} + \id$ and $S^*$ is the backward shift on $L^2_+(\T)$.
{
\item[(iii)] \textbf{Blow-up alternative:} If $T_+:=\sup I<+\infty$, then
\[
\lim_{t\nearrow T_+}\|u(t)\|_{H^{1/2}}=+\infty.
\]
Likewise, if $T_-:=\inf I>-\infty$, then
\[
\lim_{t\searrow T_-}\|u(t)\|_{H^{1/2}}=+\infty.
\]
}
\end{enumerate}
\end{proposition}

{
For completeness, let us recall the blow-up alternative in item (iii). We only discuss the forward-time statement. If it failed, there would be a sequence $t_n\nearrow T_+$ for which $\|u(t_n)\|_{H^{1/2}}$ remains bounded. The inductive coercive estimates used in the proof of \cite[Corollary~2.1]{GeLe-24} depend only on the expansion of the conserved quantities $I_k[u]$ and on Sobolev estimates that hold equally well on $\T$. They therefore imply successively that $\|u(t_n)\|_{H^{k/2}}$ remains bounded for every integer $k\geq1$. In particular, the sequence is uniformly bounded in $H^2$. Since the lifespan in the $H^2$ local well-posedness theory depends only on the size of the $H^2$ norm, the solution can then be continued beyond $T_+$, contradicting maximality. The backward-time statement follows in the same way.
}

In the explicit formula above, we note that $\eu^{-it \LL} S^* : L^2_+(\T) \to L^2_+(\T)$ is a contraction\footnote{Recall that a bounded linear operator $T: H \to H$ on a Hilbert space $H$ is a {\em contraction} if $\| T f \| \leq \| f \|$ for all $f \in H$, i.e., its operator norm satisfies $\| T \| \leq 1$.}
 for any $t \in \R$. Thus $\sigma(\eu^{-it \LL} S^*) \subset \ov{\D}$ and hence the bounded inverse $(\id-z \eu^{-it \LL}S^*)^{-1}$ exists for any $z \in \D$ and it is easy to see that {the} right-hand side in the explicit formula defines an element in $L^2_+(\T)$ for all $t\in \R$. But in order to {determine whether} $I \neq \R$, we {need to study} the strong continuity in $L^2_+(\T)$ with respect to $t$. Here, by using the general {\em stability principle} for explicit formulae derived in \cite{GeLe-26}, we obtain the following result, which characterizes the maximal time interval $I$ by the strong stability of the discrete semigroup of contractions $\{ (\eu^{-it \LL} S^*)^n \}_{n \in \N}$ on the initial datum $u_0$.

\begin{proposition}[Stability principle for \eqref{eq:CS}]  \label{prop:stability_principle}
Suppose $u_0 \in H^\infty_+(\T)$ and let $u \in C(I; H^\infty_+(\T))$ be the corresponding maximal solution  of \eqref{eq:CS} with $u(0)=u_0$. Then we have
$$
\sup I = \inf \left \{ t_* > 0 : \mbox{$(\eu^{-it_* \LL} S^*)^n u_0 \not \to 0$ in $L^2_+(\T)$ as $n \to \infty$} \right \} ,
$$
$$
\inf I = \sup \left \{ t_* < 0 : \mbox{$(\eu^{-it_* \LL} S^*)^n u_0 \not \to 0$ in $L^2_+(\T)$ as $n \to \infty$} \right \},
$$
with the usual convention that $\inf \emptyset = +\infty$ and $\sup \emptyset = -\infty$. 

Moreover, if $\sup \, I < +\infty$ then {the infimum on the right-hand side is a minimum}. Likewise, if $\inf \, I > -\infty$ then {the supremum on the right-hand side is a maximum}.

In particular, it holds that $I = \R$ if and only if we have the strong stability property
$$
\mbox{$(\eu^{-it \LL} S^*)^n u_0 \to 0$ in $L^2_+(\T)$ as $n \to \infty$ for all $t \in \R$}.
$$
\end{proposition}

\begin{remark*}
{\em Since $|I| > 0$ by local well-posedness, we conclude that strong stability $(\eu^{-it \LL} S^*)^n u_0 \to 0$ in $L^2_+(\T)$ must hold {for all sufficiently small $|t|$}.}
\end{remark*}

\begin{proof}
We organize the proof into the following steps.

\medskip
{\bf Step 1.} For $t \in \R$, we set $\Sigma_t^* := \eu^{-it \LL} S^*$ and define the map $\Us(t) : L^2_+(\T) \to L^2_+(\T)$ by
$$
(\Us(t) f )(z) := \langle (\id -z \Sigma_t^*)^{-1} f, 1 \rangle \quad \mbox{with $z \in \D$}.
$$
From {\cite[Lemma~8.1]{GeLe-26}} we recall the stability identity
\be \label{eq:stability_identity}
\| \Us(t) f \|_{L^2}^2  = \| f \|_{L^2}^2 - \lim_{n \to \infty} \| (\Sigma_t^*)^n f \|_{L^2}^2.
\ee
We remark that the limit on the right-hand side exists for all $f \in L^2_+(\T)$, since the non-negative sequence $\| f \|_{L^2} \geq \| \Sigma_t^* f \|_{L^2} \geq \| (\Sigma_t^*)^2 f \|_{L^2} \geq \ldots \geq 0$ is monotone-decreasing because $\Sigma_t^* = \eu^{-it \LL} S^*$ is a contraction.

Let us define the set
$$
J := \{ t \in \R : \mbox{$(\Sigma_t^*)^n u_0 \to 0$ in $L^2_+(\T)$ as $n \to \infty$} \}.
$$
We claim that
$$
I \subset J.
$$
We argue by contradiction. Suppose there is some $t_* \in I$ with $t_* \not \in J$, i.e., we have ${(\Sigma_{t_*}^*)^n} u_0 \not \to 0$ in $L^2_+(\T)$ as $n \to \infty$. By the explicit formula $u(t_*) = \Us(t_*) u_0$, whereas \eqref{eq:stability_identity} tells us that
$$
\| u(t_*) \|_{L^2} < \| u_0 \|_{L^2}.
$$
But this contradicts conservation of $L^2$-mass of $u \in C(I; H^\infty_+(\T))$. Hence we conclude that $I \subset J$.

{If $I=\R$, then $J=\R$ and the asserted endpoint formulas follow. The converse implication will follow from Step~2 below.}

\medskip
\textbf{Step 2.} It remains to discuss the case $I \neq \R$. Assume that
$$
T_+ := \sup \, I < +\infty
$$
is finite. We claim that
$$
\mbox{$T_+ \not \in J$ holds, i.e., we have $(\Sigma_{T_+}^*)^n u_0 \not \to 0$ in $L^2_+(\T)$ as $n \to \infty$.}
$$
We argue by contradiction and assume that $T_+ \in J$ holds. From the general discussion in \cite{GeLe-26}, we recall that
$$
\mbox{$\Us(t) u_0 \weakto \Us(T_+) u_0$ weakly in $L^2_+(\T)$ as $t \nearrow T_+$}{.}
$$
On the other hand, from $\|(\Sigma^*_{T_+})^n u_0 \|_{L^2} \to 0$ and \eqref{eq:stability_identity}, we deduce
$$
\| u_0 \|_{L^2} = \|\Us(t) u_0 \|_{L^2}  \to \| \Us(T_+) u_0 \|_{L^2}  = \| u_0 \|_{L^2} \quad \mbox{as} \quad t \nearrow T_+,
$$
using the $L^2$-mass conservation and the explicit formula $u(t) = \Us(t) u_0$ valid for $t \in I$. Therefore, we find that
$$
\mbox{$u(t)$ converges strongly in $L^2_+(\T)$ as $t \nearrow T_+$}.
$$
Take now $\eps > 0$ sufficiently small to ensure that 
$$I_\eps:=(T_+-\eps, T_+) \subset I.$$ 
By the strong convergence of $u(t)$ in $L^2_+(\T)$ as $t \nearrow T_+$,  the set $\mathcal{K} = \{ u(t) : t \in I_\eps \}$ is relatively compact in $L^2_+(\T)$ provided that $\eps > 0$ is sufficiently small. Thus we deduce the following tightness property using Fourier series: For every $\eta > 0$, there exists some integer $N \geq 1$ such that
$$
u_{\leq N}(t):= \sum_{0 \leq k \leq N} \widehat{u}_k(t) \eu^{ik x}, \quad u_{>N}(t) := \sum_{k > N} \widehat{u}_k(t) \eu^{ikx}
$$
satisfy
$$
\| u_{\leq N}(t) \|_{L^\infty}^2 \lesssim N \| u_0 \|_{L^2}^2, \quad \| u_{> N}(t) \|_{L^2}^2 \leq \eta \quad \mbox{for all $t \in I_\eps$}.
$$
Next, from \cite{Ba-24}, we recall the estimate 
$$
\| T_{\bar{g}} h \|_{L^2}^2 \leq \left ( \langle  D h,h \rangle + \| h \|_{L^2}^2 \right ) \| g \|_{L^2}^2
$$
for $g \in L^2_+(\T)$ and $h \in H^{1/2}_+(\T)$. For $t \in I_\eps$, this yields
\begin{align*}
\| T_{\bar{u}(t)} u(t) \|_{L^2}^2 & \lesssim \| T_{\bar{u}_{\leq N}(t)} u(t) \|_{L^2}^2 + \| T_{\bar{u}_{>N}(t)} u(t) \|_{L^2}^2 \\
& \lesssim  {N \| u_0 \|_{L^2}^4 + \eta\bigl(\langle D u(t),u(t)\rangle+\|u_0\|_{L^2}^2\bigr)}.
\end{align*}
 Taking $\eta > 0$ sufficiently small, the conservation of $I_1$ implies that
$$
I_1[u_0]  = I_1[u(t)] = \langle D u(t), u(t) \rangle - \| T_{\bar{u}(t)} u(t) \|_{L^2}^2  \geq \frac{1}{2} \langle D u(t), u(t) \rangle - {C_{N,u_0}}
$$
{with a constant $C_{N,u_0}>0$ independent of $t \in I_\eps$.} Hence we deduce that
\be \label{ineq:H12_bound}
\sup_{t \in I_\eps} \| u(t) \|_{H^{1/2}} < +\infty.
\ee
{The bound \eqref{ineq:H12_bound} contradicts the blow-up alternative in Proposition \ref{prop:lwp}, since $T_+=\sup I<+\infty$.}

Therefore, we conclude that $u(t)$ cannot converge strongly in $L^2_+(\T)$ as $t \nearrow T_+$. This shows that $T_+ \not \in J$ holds. Since $T_+ > 0$, we finally conclude that
$$
T_+ = \inf \{ t_* > 0 : t_* \in J^c \} = \inf \{ t_* > 0 : \mbox{${(\Sigma_{t_*}^*)^n} u_0 \not \to 0$ in $L^2_+(\T)$ as $n \to \infty$} \}.
$$
Finally, since we have shown that $T_+ \in J^c$, we see that the infimum is a minimum. The proof for the case with  $T_- = \inf I > -\infty$ is analogous. 

{This} finishes the proof of Proposition \ref{prop:stability_principle}.
\end{proof}

\section{Finite-gap potentials}

\label{sec:finite_gap}

This section collects some important facts about finite-gap potentials for \eqref{eq:CS}. Here it suffices to treat the case where $u_0 \in H^\infty_+(\T)$ is a finite-gap potential {which is a rational function} of the form \eqref{eq:finite_gap}. {For the remaining trivial case of plane-wave profiles $u_0(z) = C z^N$, see the remarks below.} Let us write
\be \label{eq:finite_gap_3}
u_0(z) = z^{m_0} \beta_{\bm{p}}(z) \left ( a + \sum_{j=1}^r \frac{c_j}{1- \bar{p}_j {z}} \right ),  \quad \beta_{\bm{p}}(z) := \prod_{j=1}^r \left ( \frac{z-p_j}{1- \bar{p}_j z} \right )^{m_j-1},
\ee
where $m_0 \geq 0$ and $\bm{p} = (p_1, \ldots, p_r) \in (\D \setminus \{ 0 \})^r$ with $p_j \neq p_k$ for $j \neq k$ and, for some $N \in \N$, $0 \leq m_0 \leq N-1$, $1 \leq m_1, \ldots, m_r \leq N$ such that
$m_0 + \sum_{j=1}^r m_j = N$, and the parameters $a, c_1, \ldots, c_r \in \C$ satisfy the constraints
\be \label{eq:constraints_fg}
{\bar{a} c_j + \sum_{k=1}^r \frac{c_j \bar{c}_k}{1- \bar{p}_j p_k} = m_j}
\ee
for $j=1, \ldots, r$ and $a \neq 0$ if $m_0 \neq 0$. We define the inner function\footnote{Recall that a bounded holomorphic function $\psi : \D \to \D$ is said to be {\em inner} if $|\psi(\zeta)|=1$ for a.e.~$\zeta \in \pt\D$. Note that {the} existence of boundary values a.e.~on $\pt \D$ of a bounded holomorphic {function} on $\D$ is a standard fact.}
\be
\psi_{u_0}(z) := z^{m_0} \prod_{j=1}^r \left  (\frac{z-p_j}{1-\bar{p}_j z} \right )^{m_j} \quad \mbox{for} \quad z \in \D,
\ee
which is a finite Blaschke product with zero $0$ (if $m_0 \geq 1$) and distinct {zeros} $p_1, \ldots, {p_r} \in \D \setminus \{ 0 \}$ with corresponding multiplicities $m_1, \ldots, {m_r} \geq 1$. Correspondingly, its degree is given by
$\mathrm{deg} \, \psi_{u_0} = m_0 + \sum_{j=1}^r m_j = N$. Note also that $\psi_{u_0}(z)$ has all its poles outside $\ov{\D}$ and thus it extends smoothly to the unit circle $\pt \D$. 

\subsection{Spectral properties and invariant subspaces}
 From {\cite[Section~5]{Ba-25}} we recall that, by using the finite-gap potential property of $u_0$ and in particular the constraints {\eqref{eq:constraints_fg}}, we obtain that
\be
L_{u_0} S^k \psi_{u_0} = (\nu_{u_0}+k) S^k \psi_{u_0} \quad \mbox{for $k \in \Z_{\geq 0}$}
\ee
with some $\nu_{u_0} \in \R$, where we recall that $S$ denotes the forward shift. Hence $\psi_{u_0}$ is an eigenfunction of $L_{u_0}$ with eigenvalue $\nu_{u_0}$ and the orthonormal set $\{ S^k \psi_{u_0} : k \in \Z_{\geq 0} \}$ is invariant under $L_{u_0}$ and $S$. {Since these eigenfunctions span $\psi_{u_0}L^2_+(\T)$, this closed subspace reduces the self-adjoint operator $L_{u_0}$; moreover, its invariance under $S$ implies that its orthogonal complement is invariant under $S^*$. Consequently,} the subspace
$$
W_0 := (\psi_{u_0} L^2_+)^\perp = \{ h \in L^2_+(\T) : \mbox{$\langle h, \psi_{u_0} g \rangle = 0$ for all $g \in L^2_+(\T)$} \}
$$
is invariant under $S^*$ and $L_{u_0}$. We remark that $W_0$ is the so-called {\em model space} generated by the inner function $\psi_{u_0}$. Moreover, since $\psi_{u_0}$ is a finite Blaschke product with $\mathrm{deg} \, \psi_{u_0} = N$, it is {well known that} $\dim W_0 = N$ holds; see, e.g., \cite{GaMaRo}. Let $\nu_{0}, \ldots, \nu_{N-1}$ denote the corresponding eigenvalues of $L_{u_0} |_{W_0}$. For later purposes, it turns out to be convenient to enlarge $W_0$ by one dimension by introducing the slightly {larger} model space
\be \label{eq:W_ortho}
W := (z \psi_{u_0} L^2_+)^\perp = \C \psi_{u_0} \oplus W_0,
\ee
where the orthogonal decomposition on the right-hand side is easily checked; see, e.g., {\cite[Lemma~5.10]{GaMaRo}}. Also, note that $W$ is spanned by eigenfunctions of $L_{u_0}$, whence it follows that $W \subset H^1_+(\T)$, where in fact we have $W\subset H^\infty_+(\T)$ since these {functions} are rational.

\begin{proposition} \label{prop:invariant}
The subspace $W$ has $\dim W = N+1$ and is invariant under $L_{u_0}$ and $S^*$. Moreover, we have $u_0 \in W$ and $1 \in W$.
\end{proposition}

\begin{remark*}
{\em For finite-gap potentials given by the trivial plane wave profiles $u_0(z) =C z^N$ with ${C\in \C}$ and some integer $N \geq 0$, we remark that the invariant subspace simply becomes
$W = (z^{N+1} L^2_+)^\perp = \mathrm{span} \{1, z, z^2, \ldots, z^N \}$. Recall that we here consider the non-trivial case where $u_0 \in H^\infty_+(\T)$ is a finite-gap potential of rational type as in \eqref{eq:finite_gap}.
}
\end{remark*}

\begin{proof}
From \eqref{eq:W_ortho} and $\dim W_0 = N$, we see that $\dim W = N+1$. Since $W$ is a model space, it is invariant under $S^*$. Because $W_0$ is invariant under $L_{u_0}$ and $\psi_{u_0}$ is an eigenfunction of $L_{u_0}$, it follows that $L_{u_0}(W) \subset W$. Next, we note that $1 \perp z L^2_+(\T)$ and hence $1 \in W$. To show that $u_0 \in W$ holds, let $g \in L^2_+(\T)$ be given and set $f(z) = a + \sum_{j=1}^r \frac{c_j}{1- \bar{p}_j z}$ so that $u_0(z) = z^{m_0} \beta_{\bm{p}}(z) f(z)$. Since $\bar{z} = z^{-1}$ on $\T$, we find
\begin{align*}
\langle u_0, z \psi_{u_0} g \rangle &  = \left \langle f, z \phi_{\bm{p}} g  \right \rangle \quad \mbox{with} \quad \phi_{\bm{p}}(z) := \prod_{j=1}^r \left ( \frac{z-p_j}{1- \bar{p}_j z} \right ) \\
& =  \left \langle a, z  \phi_{\bm{p}} g \right \rangle + \sum_{j=1}^r  \left \langle \frac{c_j}{1-\bar{p}_j z}  \ov{\phi_{\bm{p}}}, z  g\right \rangle =0 + \sum_{j=1}^r 0 = 0,
\end{align*}
using that $1 \perp z L^2_+$ and $\frac{1}{1-\bar{p}_j z} \ov{\phi_{\bm{p}}(z)} = \frac{ \prod_{\ell \neq j}(1-\bar{p}_\ell z)}{\prod_{\ell = 1}^r (z-p_{\ell})}\perp L^2_+$, since $p_\ell \in \D$ for $\ell =1, \ldots, r$.
\end{proof}

\subsection{Reduction of explicit formula to $W$}

Assume now that $u_0 \in H^\infty_+(\T)$ is a finite-gap potential of the form \eqref{eq:finite_gap_3}. Hence $L_{u_0}$ and $S^*$ act invariantly on the finite-dimensional subspace $W$ with $1, u_0 \in W$. For the corresponding maximal solution $u \in C(I; {H^\infty_+(\T)})$ with $u(0)=u_0$, the explicit formula \eqref{eq:EF} for $u(t,z)$ reduces to $W$. That is, we can write
\be
u(t,z) = \langle \left (\id-z A_t)^{-1} u_0, 1 \right \rangle \quad \mbox{for $(t,z) \in I \times \D$},
\ee
with the time-dependent family of contractions
$$
A_t := \eu^{-it \LL} S^* |_W \quad \mbox{with} \quad t \in \R
$$
obtained by restriction of  $\LL = 2L_{u_0} + \id$ and $S^*$ to the finite-dimensional invariant subspace $W$. 

We will now adapt the general {\em stability principle} for {explicit formulae} derived in \cite{GeLe-26} to the {present setting}. By the finite-dimensionality of $W$, the spectrum of $A_t$ only contains eigenvalues, {and hence the analysis simplifies substantially}. Let $r(A_t)$ denote the {\em spectral radius} of $A_t$ for $t \in \R$, i.e., we set
$$
r(A_t) := \sup \{ |\lambda| : \lambda \in \sigma(A_t) \}.
$$
Since $A_t$ acts on the finite-dimensional space $W$, we have that 
$$
r(A_t) = \max \{ |\lambda| : \mbox{$\lambda \in \C$ {is an eigenvalue} of $A_t$} \}.
$$
Furthermore, since {$A_t$ is a contraction for all $t \in \R$}, we clearly have that
$$
r(A_t) \leq 1 \quad \mbox{for any $t \in \R$}.
$$ 
By using the circle of ideas from the general discussion in \cite{GeLe-26}, we obtain the following result.

\begin{lemma}[Blow-up criterion for finite-gap potentials] \label{lem:blow_up_criterion}
Suppose $u_0 \in H^\infty_+(\T)$ is a finite-gap potential and let $u \in C(I; H^\infty_+(\T))$ denote the corresponding maximal solution of \eqref{eq:CS} with $u(0)=u_0$. Then the following holds true.
\begin{itemize}
\item[(i)] \textbf{Global existence:} $I = \R$ holds if and only if $r(A_t) < 1$ for all $t \in \R$.
\item[(ii)] \textbf{Finite-time blow-up:} $I \neq \R$ if and only if $r(A_{t_*}) = 1$ for some $t_* \in \R$. 
\item[(iii)] \textbf{Characterization of maximal time interval}:  We have
$$
\sup I = \inf \left \{ t_* > 0 : r(A_{t_*}) = 1 \right \} ,
$$
$$
\inf I = \sup \left \{ t_* < 0 : r(A_{t_*}) = 1 \right \},
$$
with the usual convention that $\inf \emptyset = +\infty$ and $\sup \emptyset = -\infty$.
\end{itemize}
\end{lemma}

\begin{remarks*}
{\em 1) By local well-posedness and (i) above, we deduce that $r(A_t) < 1$ for {$0<|t|<\delta$ with some $\delta>0$}.

{2) The sufficiency of the spectral-radius condition in (ii) relies on the finite-dimensionality of $W$. Indeed, on the full space $L^2_+(\T)$ one has
\[
r(S^*)=1,
\]
although $(S^*)^nf\to0$ for every $f\in L^2_+(\T)$. Thus the spectral-radius criterion is not valid in this form on the full Hardy space.}}
\end{remarks*}

\begin{proof}
Let $u_0 \in H^\infty_+(\T)$ be a finite-gap potential. We omit the trivial case when $u(z) = C z^N$ is a traveling plane wave profile and assume that $u_0 \in H^\infty_+(\T)$ is of the form {\eqref{eq:finite_gap_3}} in what follows.

In view of Proposition \ref{prop:stability_principle} and since $r(A_t) \leq 1$, it remains to show the equivalence 
\be
\mbox{$A_t^n u_0 \to 0$ in $L^2_+(\T)$ as $n \to \infty$} \quad \Longleftrightarrow \quad r(A_t) < 1
\ee
for any $t \in \R$.

\medskip
{\bf Step 1.} Assume that $r(A_t) < 1$ holds for some $t \in \R$. We recall Gelfand's formula $r(T) = \lim_{n \to \infty} \| T^n \|^{1/n}$ for any bounded linear operator $T$. Let us pick $\rho>0$ such that $r(A_t) < \rho < 1$. Take now $n_0 \geq 1$ sufficiently large such that $\|A_t^n \|^{1/n} \leq \rho$ for $n \geq n_0$. Hence $\| A_t^n \| \leq \rho^n$ for $n \geq n_0$. Since $\rho < 1$, we deduce that $\| A_t^n \| \to 0$ as $n \to \infty$. In particular, we deduce that 
$$
\mbox{$A_t^n u_0 \to 0$ in $L^2_+(\T)$ as $n \to \infty$}. 
$$

\medskip
{\bf Step 2.} Suppose that $A_t^n u_0 \to 0$ in $L^2_+(\T)$ as $n \to \infty$ for some $t \in \R$. We argue by contradiction and assume $r(A_t)=1$ holds.  Let us denote $A=A_{t}$ for notational simplicity. Since $A : W \to W$ acts on the finite-dimensional space $W$, we deduce that $A$ has a unimodular eigenvalue $\omega \in \pt \D$. Let $f \in W$ be a normalized eigenfunction with $A f = \omega f$. By Lemma \ref{lem:Wu} below, we find that
$|\langle u_0, f \rangle | = 1$. Thus we can decompose $u_0 \in W$ as
$$
u_0 = \alpha f + g \quad \mbox{with} \quad  \alpha \neq 0 \quad \mbox{and} \quad g \perp f.
$$
We now claim that this implies  
\be \label{eq:contradiction}
\mbox{$A^n u_0 \not \to 0$ in $L^2_+$ as $n \to \infty$},
\ee
which would lead to the desired contradiction.  To show the claim above, we first note that $A^n g \perp f$ for all $n \geq 0$. Indeed, from $\| A f \|_{L^2} = |\omega| \| f \|_{L^2} = \| f \|_{L^2}$ we deduce that 
$$
{\left \langle (\id - A^* A)f, f\right \rangle = \|f \|_{L^2}^2 - \| A f \|_{L^2}^2=0.}
$$
Since $A$ is a contraction, we have that $\id - A^* A \geq 0$. Hence $(\id - A^* A) f = 0$ and therefore $A^* A f = f$. Using now that $A f = \omega f$, this shows ${f = A^*(\omega f) = \omega A^*f}$, whence it follows that $A^* f = \overline{\omega} f$. Therefore, for any $g \perp f$ and $n \geq 0$, we obtain $\langle A^n g, f \rangle = \langle g, (A^*)^n f \rangle = \omega^n \langle g, f \rangle = 0$. This shows that $A^n g \perp f$ for all $n \geq 0$. Now, we deduce
$$
{A^n u_0 = \alpha \omega^n f + A^n g} \quad \mbox{with $A^n g \perp f$ for all $n \geq 0$}.
$$
Therefore, we see $\| A^n u_0\|_{L^2}^2 = |\alpha|^2 + \| A^n g \|_{{L^2}}^2 \not \to 0$ since $\alpha \neq 0$. Thus we arrive at the desired contradiction in \eqref{eq:contradiction}.
\end{proof}

\section{Universal blow-up dynamics for finite-gap potentials}

This section is devoted to the proof of Theorem \ref{thm:universal}. Throughout the following, we suppose that $u_0 \in H^\infty_+(\T)$ is a finite-gap potential and that the corresponding solution $u \in C([0,T); H^\infty_+(\T))$ blows up at finite time $T> 0$. Since the trivial cases $u_0(z) = C z^N$ with $C \in \C$ and integer $N \geq 0$ yield plane wave solutions of \eqref{eq:CS}, which evidently do not lead to finite-time blow-up, we henceforth assume that $u_0 \in H^\infty_+(\T)$ is a finite-gap potential of the form \eqref{eq:finite_gap_3}. 

For the reader's convenience, we recall that the finite-dimensional model space 
$$
W = (z \psi_{u_0} L^2_+)^\perp \quad \mbox{with} \quad \dim W = N+1
$$
is invariant under $L_{u_0}$ and $S^*$, and we have that $u_0 \in W$ and $1 \in W$. Thus the explicit formula for the finite-time blow-up solution $u \in C([0,T); H^\infty_+(\T))$ reduces to $W$ with
$$
u(t,z) = \langle (\id-z A_t)^{-1} u_0, 1 \rangle \quad \mbox{for $(t,z) \in [0,T) \times \D$}
$$
with the family of contractions $A_t : W \to W$ given by
$$
A_t = \eu^{-it \LL} {S^*|_W} \quad \mbox{for $t \in \R$}
$$
where $\LL = 2 L_{u_0}+ \id$. By assumption, the solution $u(t)$ blows up as $t \nearrow T$ and hence, by Lemma \ref{lem:blow_up_criterion},
$$
T = \min \{ t_* > 0 : r(A_{t_*}) = 1 \}.
$$
{Since $A_T$ acts on the finite-dimensional space $W$, it follows that $A_T$ has at least one unimodular eigenvalue. We denote its unimodular eigenvalues, repeated according to their multiplicities, by $\omega_1, \ldots, \omega_M \in \pt \D$.} Below we will derive a universal upper bound on $M \geq 1$ along with simplicity of the unimodular eigenvalues of $A_T$. Also, since $A_T$ is a contraction, it is easy to see that we obtain the orthogonal decomposition
$$
W = W_{\mathrm{u}} \oplus W_{\mathrm{u}}^\perp{.}
$$
{Here $W_{\mathrm{u}}$ denotes the unitary part of $A_T$. Both $W_{\mathrm{u}}$ and $W_{\mathrm{u}}^\perp$ are invariant subspaces of $A_T$.} Since $A_T |_{W_{\mathrm{u}}}$ is unitary, we can always find an orthonormal set of eigenvectors $f_1, \ldots, f_M \in W$ that form a basis for $W_{\mathrm{u}}$.

\subsection{Spectral properties of $A_T$}

We first show the following identity, which will be essential for our blow-up analysis of $u(t,z)$ as $t \nearrow T$. We remark that the following virial identity is reminiscent of the so-called `Wu-type' identity obtained for eigenfunctions of the Lax operator for many integrable PDEs with a Hardy space structure. Here we obtain an identity related to the eigenfunctions that span the unitary part of $A_t$.
\begin{lemma}[Virial identity for the unitary part of $A_t$] \label{lem:Wu}
Let $T \in \R$ be given. Suppose that $f\in W$ and $\omega \in \C$ satisfy
$$
A_T f = \omega f  \quad \mbox{with} \quad \|f \|_{L^2}=1 \quad \mbox{and} \quad |\omega| = 1.
$$
Then {we have}
$$
\boxed{|\langle  u_0,f \rangle| = 1}{.}
$$
Moreover, the unimodular eigenvalue $\omega$ for $A_T$ is simple.
\end{lemma}

\begin{remarks*}{\em
1) {The identity above extends beyond finite-gap potentials. More precisely, for any $u_0\in L^2_+(\T)$, every normalized vector $f$ satisfying $\eu^{-it\LL}S^*f=\omega f$ with $|\omega|=1$ obeys $|\langle u_0,f\rangle|=1$; see \cite{ChLe-26}.} However, the proof here is much simpler as it avoids any domain issues and approximation arguments, since in the general case it is not guaranteed that $f \in \dom(L_{u_0})$ holds.

2) In the proof below, we will first show that $\dim \ker (A_T - \omega \id) = 1$ holds, i.e., the unimodular eigenvalue $\omega$ has geometric multiplicity $m_{\mathrm{geom}}(\omega)=1$. {To prove simplicity of $\omega$, we must also show that its algebraic multiplicity satisfies $m_{\mathrm{alg}}(\omega)=1$. For this we use the fact that $A_T$ is a contraction and that $|\omega|=1$.} This simplicity of the unimodular eigenvalues of $A_T$ will be essential for the perturbation analysis below.}
\end{remarks*}

\begin{proof}
We divide the proof into two steps as follows.

\medskip
\textbf{Step 1.} For notational convenience, we write $u= u_0$ and $L_u = L_{u_0}$. Since $\LL = 2 L_u + \id$, the assumption $A_T f = \omega f$  yields that
\be
\eu^{-i \tau L_u} S^* f = \lambda f
\ee
with {$\tau = 2T$ and the unimodular eigenvalue $\lambda = \eu^{iT} \omega \in \pt \D$}. Since $\eu^{-i \tau L_u}$ is unitary, 
$$
1 = \| \eu^{-i \tau L_u} S^* f \|_{L^2}^2 = \| S^* f \|_{L^2}^2 = \| f \|_{L^2}^2 - |\langle f, 1 \rangle |^2.
$$
Since $\| f \|_{L^2}^2=1$, this shows that $f \perp 1$ and therefore $S S^* f = f$. Because $S^* f = \lambda \eu^{i \tau L_u} f$ and the unitary map $\eu^{i \tau L_u}$ commutes with $L_u$, we find
\be \label{eq:Wu1}
\langle L_u S^*f, S^* f \rangle =  \langle L_u f, f \rangle,
\ee
using also that $|\lambda| = 1$.

On the other hand, from \cite{Ba-24} we recall the commutator identity $[S^*, L_u] = S^* - \langle \cdot , u \rangle S^* u$ on $H^1_+(\T)$. Since $f \in W \subset H^1_+(\T)$, it follows
$$
L_u S^* f = S^* L_u f - S^* f +\langle f, u \rangle S^*u{.}
$$
Taking the inner product with $S^* f$, we obtain
$$
\langle L_u S^*f, S^*f  \rangle = \langle S^* L_u f, S^* f \rangle - \| S^* f \|_{L^2}^2 + \langle f,u \rangle \langle S^*u, S^*f \rangle.
$$
Since $S S^*f = f$, we find
$$
\langle S^* L_u f, S^*f \rangle = \langle L_u f,f \rangle, \quad \langle S^* u, S^* f \rangle = \langle u, f \rangle.
$$
In view of \eqref{eq:Wu1} and $\| S^* f \|_{L^2}^2 = 1$, we deduce 
$$
\langle L_u f, f \rangle = \langle L_u f,f \rangle - 1 + |\langle u,f \rangle|^2,
$$
which shows that $|\langle u,f \rangle| = 1$ holds.

\medskip
\textbf{Step 2.}  Set $V = \ker ( A_T - \omega \id)$ and define the linear form $\ell : V \to \C$ with $\ell(f) := \langle  f,u \rangle$. From {the identity obtained in the previous step}, we deduce that $|\ell(f)| = 1$ for all $f \in V$ with $\| f \|_{L^2}=1$. This shows that $\ker \, \ell = \{ 0 \}$ is {trivial}, which implies that $\dim V =1$ holds. Thus the geometric multiplicity of $\omega$ is equal to one.

To prove that the algebraic multiplicity of $\omega$ is also one, we must rule out nontrivial Jordan blocks. We argue by contradiction and {suppose that there exists $h \in W$ such that}
$$
(A_T - \omega \id) h = f.
$$
By iteration, we deduce that 
$$
A_T^n h = \omega^n h + n \omega^{n-1} f \quad \mbox{for $n \geq 1$}.
$$
Since $|\omega|=1$, the right-hand side has norm growing at least linearly in $n$. On the other hand, because $A_T$ is a contraction, we must have $\| A_T^n h \| \leq \| h \|$ for any $n \geq 1$, which leads to a contradiction. {Thus the algebraic multiplicity of $\omega$ is one.}

This completes the proof of Lemma \ref{lem:Wu}.
\end{proof}

As a consequence of Lemma \ref{lem:Wu}, we obtain the following fact.

\begin{lemma} \label{lem:simple}
Assume that $f_1, \ldots, f_M \in W$ {form} an orthonormal set of eigenvectors for $A_T$ with corresponding unimodular eigenvalues $\omega_1,\ldots, \omega_M$, i.e., we have $\langle f_j, f_k \rangle = \delta_{jk}$ and
$$
A_T f_j = \omega_j f_j \quad \mbox{and} \quad |\omega_j| = 1 \quad \mbox{for} \quad j=1, \ldots, M.
$$ 
Then we have
$$
\boxed{1 \leq M \leq \| S^* u_0 \|_{L^2}^2 \leq \| u_0 \|_{L^2}^2} \quad \mbox{and} \quad \boxed{\omega_j \neq \omega_k \quad \mbox{if} \quad j \neq k}{.}
$$
\end{lemma}

\begin{proof}
As usual, we write $u=u_0$ for notational simplicity. From the proof of Lemma \ref{lem:Wu} we know that $f_j \perp 1$ and thus $S S^* f_j = f_j$ for $j=1, \ldots, M$. Thus $\langle S^* f_j, S^*  f_k \rangle = \langle f_j, f_k \rangle = \delta_{jk}$ and the vectors $S^* f_1, \ldots, S^* f_M$ form an orthonormal set. Since  $|\langle f_j, u \rangle| = |\langle S^* f_j, S^* u \rangle| = 1$ by Lemma \ref{lem:Wu}, we apply Bessel's inequality to deduce
 $$
 M = \sum_{j=1}^M |\langle f_j, u \rangle|^2 = \sum_{j=1}^M |\langle S^* f_j, S^* u \rangle|^2 \leq \| S^* u \|_{L^2}^2 \leq \|u \|_{L^2}^2,
 $$
 where the last inequality is obviously true, since $S^*$ is a contraction. Finally, the assertion that $\omega_j \neq \omega_k$ for $j \neq k$ follows from $f_j \perp f_k$ for $j \neq k$ and the simplicity of unimodular eigenvalues in Lemma \ref{lem:Wu}.
\end{proof}

For later use, we also show that  $u_0$ cannot entirely lie in the unitary subspace $W_{\mathrm{u}} = \mathrm{span} \, \{f_1, \ldots, f_M \}$ of $A_T$. As a consequence, the inequality in Lemma \ref{lem:simple} for the dimension $M$ in Lemma \ref{lem:simple} can be sharpened.

\begin{lemma} \label{lem:no_total_loss}
It holds that $u_0 \not \in W_{\mathrm{u}}$. Consequently, we have that $1 \leq M < \| u_0 \|_{L^2}^2$ holds in Lemma \ref{lem:simple}.
\end{lemma}

\begin{proof}
Let us divide the proof into the following steps.

\medskip
\textbf{Step 1.} Note that 
\be \label{eq:W_u_perp}
W_{\mathrm{u}} \perp 1.
\ee
Indeed, for any $f \in W_{\mathrm{u}}$ and using that $\eu^{-iT \LL}$ is unitary, we deduce that
$$
\| f \|_{L^2}^2 = \| A_T f \|_{L^2}^2 = \| \eu^{-iT \LL} S^* f \|_{L^2}^2 = \| S^* f \|_{L^2}^2 =  \| f \|_{L^2}^2 - |\langle f, 1 \rangle|^2,
$$
which shows \eqref{eq:W_u_perp}.  

We now assume, towards a contradiction, that $u_0 \in W_{\mathrm{u}}$ holds. By \eqref{eq:W_u_perp} and $u_0 \neq 0$, there exists some integer $m \geq 1$ such that
\be \label{eq:u_0_fourier}
 {u_0(z) = \sum_{k \geq m} a_k z^k \quad \mbox{with} \quad a_{m} \neq 0.}
\ee
We now claim that
\be \label{eq:W_u_contra}
\langle A_T^{m} u_0, 1 \rangle \neq 0.
\ee
{On the other hand, since $W_{\mathrm{u}}$ is invariant under $A_T$, we have $A_T^k u_0\in W_{\mathrm{u}}$ for every $k\geq0$, which contradicts \eqref{eq:W_u_perp} in view of \eqref{eq:W_u_contra}.} Therefore, it remains to show that \eqref{eq:W_u_contra} holds.

\medskip
\textbf{Step 2.} For $m \geq 1$ as in \eqref{eq:u_0_fourier}, we introduce the $m$-dimensional subspace
$$
E_m = \mathrm{span} \{ 1, \ldots, z^{m-1} \}.
$$
 Clearly,  the subspace $E_m$ is invariant under $D = -i\pt_x = z \pt_z$. Since $u_0(z)= z^m v(z)$ with some $v \in L^2_+(\T)$ by \eqref{eq:u_0_fourier}, we find 
$$
T_{\bar{u}_0} f = \Pi( \bar{z}^m \bar{v} f) = 0 \quad \mbox{for $f \in E_m$},
$$
because $\bar{z}^m \bar{v} f$ has only strictly negative Fourier modes for $f \in E_m$. Therefore, the subspace $E_m$ is also invariant under $L_{u_0} = D-T_{u_0} T_{\bar{u}_0}$ with 
$$
L_{u_0} f = D f \quad \mbox{for} \quad f \in E_m.
$$ 
{Since $L_{u_0}$ is self-adjoint and preserves smoothness,} we see that $E_m^\perp \cap H^\infty_+(\T)$ is also invariant under $L_{u_0}$. Recalling that $\LL = 2 L_{u_0} + \id$, we hence find
$$
P_m \eu^{-iT \LL} h = \eu^{-iT (2D+1)} P_m h \quad \mbox{for} \quad h \in H^\infty_+(\T),
$$
where $P_m$ denotes the orthogonal projection onto $E_m$.

For $0\leq k\leq m$, let
$$
h_k:=A_T^k u_0.
$$
We claim that
\be \label{eq:lowest_mode}
\widehat{h_k}(n)=0 \quad \mbox{for $0\leq n<m-k$},
\quad
\widehat{h_k}(m-k)
=
\eu^{-iT\sum_{\ell=1}^k(2(m-\ell)+1)}a_m
\neq 0,
\ee
where the sum is understood to be zero when $k=0$. Indeed, the assertion for $k=0$ follows from \eqref{eq:u_0_fourier}. {Suppose that the assertion holds with $k-1$ in place of $k$, where $1\leq k\leq m$.} Since $S^*$ shifts every positive Fourier mode down by one, $S^*h_{k-1}$ has no Fourier modes below $m-k$, and
$$
\widehat{S^*h_{k-1}}(m-k)
=
\widehat{h}_{k-1}(m-k+1).
$$
The identity
$$
P_m\eu^{-iT\LL}h
=
\eu^{-iT(2D+1)}P_mh
$$
then shows that $h_k=\eu^{-iT\LL}S^*h_{k-1}$ has no Fourier modes below $m-k$ and that
$$
\widehat{h}_k(m-k)
=
\eu^{-iT(2(m-k)+1)}
\widehat{h}_{k-1}(m-k+1).
$$
This proves \eqref{eq:lowest_mode} by induction. Taking $k=m$ and using $\sum_{\ell=1}^m(2(m-\ell)+1)=m^2$, we obtain
$$
\langle A_T^m u_0,1\rangle
=
\widehat{h}_m(0)
=
\eu^{-iTm^2}a_m
\neq0.
$$
This proves \eqref{eq:W_u_contra} and complete the proof that $u_0 \not \in W_{\mathrm{u}}$.

Finally, write $u_0 = f + g$ with $f \in W_{\mathrm{u}}$ and $g \in W_{\mathrm{u}}^\perp$. Since $g \neq 0$ and $\| f \|_{L^2}^2 = \sum_{j=1}^M |\langle f_j, u_0 \rangle |^2 = M$, we obtain $\| u_0 \|_{L^2}^2 = M+ \|g \|_{L^2}^2 > M$. This completes the proof of Lemma \ref{lem:no_total_loss}.
\end{proof}

\subsection{Blow-up analysis by analytic perturbation theory} \label{subsec:blowup_analysis}
Recall that $u_0 \in H^\infty_+(\T)$ is a finite-gap potential of the form \eqref{eq:finite_gap} and we assume that the corresponding solution $u \in C([0,T); H^\infty_+(\T))$ blows up at finite time $T >0$. For $\eps \in \R$, we consider the family of contractions on $W$ given by
\be
\AA(\eps) := A_{T-\eps} : W \to W.
\ee
Since $A_t = \eu^{-it \LL} S^*|_W$ and by invariance of $W$ under $\LL$ and $S^*$ together with the group property $\eu^{-i(t+s) \LL} = \eu^{-is \LL} \eu^{-it \LL}$, we can write
\be
\AA(\eps) = \eu^{i \eps \Ls} \AA_0,
\ee
where we set
$$
\Ls := \LL |_W \quad \mbox{and} \quad \AA_0 := A_T.
$$
Since $W$ is finite-dimensional, we see that 
$$
\AA(\eps) = \sum_{n=0}^\infty \frac{(i \eps)^n}{n!} {\Ls^n} \AA_0 = \AA_0 + \sum_{n=1}^\infty \frac{(i \eps)^n}{n!} {\Ls^n} \AA_0
$$
 is an analytic family of operators on the finite-dimensional Hilbert space $W$ and hence standard analytic perturbation theory for matrices can be applied; see, e.\,g.,~\cite{Ka-95}.
  
 Recall from above that $\AA_0 = A_T$  has simple unimodular eigenvalues $\omega_1, \ldots, \omega_M \in \pt \D$ and let  $f_1, \ldots, f_M \in W$ denote a fixed orthonormal system of eigenvectors for $\AA_0$. Thanks to simplicity of these eigenvalues, we can apply standard perturbation theory to construct {analytic branches} of normalized eigenvectors $f_j(\eps) \in W$ and corresponding eigenvalues $\omega_j(\eps)$ for $\AA(\eps)$, which emanate from $f_j(0)=f_j$ and $\omega_j(0) = \omega_j$, {provided that $|\eps|<\eps_0$ for some $\eps_0>0$}. Here is the statement, where $\AA(\eps)^* : W \to W$ denotes the adjoint of $\AA(\eps) : W \to W$.
 
\begin{proposition} \label{prop:perturbation}
For some $\eps_0 > 0$ sufficiently small, the following properties hold.
\begin{enumerate}
\item[(i)] There exist functions $f_j : (-\eps_0, \eps_0) \to W$ and $\omega_j : (-\eps_0, \eps_0) \to \ov{\D}$ for $j=1, \ldots, M$ such that
$$
\AA(\eps) f_j(\eps) = \omega_j(\eps) f_j(\eps) \quad \mbox{with} \quad \| f_j(\eps) \|_{L^2} = 1,
$$
where $f_j(\eps)$ and $\omega_j(\eps)$ are analytic and satisfy 
$$
f_j(0) = f_j \quad \mbox{and} \quad \omega_j(0) = \omega_j \quad \mbox{for $j=1, \ldots, M$}.
$$
\item[(ii)] There exist analytic functions $g_1, \ldots, g_M : (-\eps_0, \eps_0) \to W$ such that
$$
\AA(\eps)^* g_j(\eps) = \ov{\omega}_j(\eps) g_j(\eps) \quad \mbox{with} \quad \langle f_j(\eps), g_j(\eps) \rangle = 1
$$
 and $g_j(0) = f_j$. Correspondingly, the analytic maps $P_j(\eps) : W \to W$ with
 $$
 P_j(\eps) h = \langle h, g_j(\eps) \rangle f_j(\eps)
 $$
 are the Riesz projections onto the one-dimensional eigenspaces $\ker (\AA(\eps) - \omega_j(\eps) \id)$ for $\eps \in (-\eps_0, \eps_0)$.

\item[(iii)] $\AA(\eps)$ has no unimodular eigenvalue for $\eps \in (-\eps_0, \eps_0)$ with $\eps \neq 0$, i.e.,
$$
|\omega_j(\eps)| < 1 \quad \mbox{for $|\eps| < \eps_0$ with $\eps \neq 0$}.
$$
The remaining spectrum of $\AA(\eps)$ is uniformly bounded away from $\pt \D$, i.e.,
$$
 {\sigma(\AA(\eps)) \setminus \{ \omega_1(\eps), \ldots, \omega_M(\eps) \} \subset \{  |z|  \leq \rho \}} 
$$
for $\eps \in (-\eps_0, \eps_0)$ with some constant $0 < \rho < 1$.
\end{enumerate}
\end{proposition}

\begin{remark*}
{\em By `analytic' we mean {\em real analytic}, i.e., the corresponding functions are given by a convergent power series in $\eps \in \R$ {with $|\eps|<\eps_0$ for some $\eps_0>0$}. Actually, some functions below are in fact complex analytic, but we won't need this fact here.}
\end{remark*}

\begin{proof}
{This follows from standard perturbation theory for simple eigenvalues of the analytic family $\AA(\eps)$ on the finite-dimensional space $W$; see, e.g.,~\cite{Ka-95}.} Note also that $|\omega_j(\eps)| \leq 1$ holds, since $\AA(\eps)$ is a contraction for any $\eps$ and hence its eigenvalues must be contained in $\ov{\D}$. This proves (i).

{The proof of (ii) is likewise a direct consequence of standard perturbation theory, with $g_j(\eps)$ denoting the corresponding normalized eigenvector of the adjoint family.}

As for (iii), {after decreasing $\eps_0$ if necessary, we may assume that $\eps_0<T$. If $0<\eps<\eps_0$, then $0<T-\eps<T$, and the solution is smooth at $t=T-\eps$. Hence the blow-up criterion implies that $\AA(\eps)=A_{T-\eps}$ has no unimodular eigenvalue, so $|\omega_j(\eps)|<1$ for $0<\eps<\eps_0$. The nonnegative function $1-|\omega_j(\eps)|^2$ is real analytic and is strictly positive for $0<\eps<\eps_0$. It is therefore not identically zero, and its zero at $\eps=0$ is isolated. After decreasing $\eps_0$ if necessary, this gives $|\omega_j(\eps)|<1$ also for $-\eps_0<\eps<0$.} Also, {all remaining eigenvalues} of $\AA(\eps)$ for $|\eps| < \eps_0$ must belong to $\ov{\D}$, since $\AA(\eps)$ is a contraction. In fact, they are uniformly bounded away from $\pt \D$. Indeed, let $\chi_\eps(z)$ denote the characteristic polynomial of $\AA(\eps)$. Then we have the factorization
$$
\chi_\eps(z) = \prod_{j=1}^M (z- \omega_j(\eps)) \cdot \tilde{\chi}_\eps(z) \quad \mbox{for $|\eps| < \eps_0$}, 
$$
where the zeros of $\tilde{\chi}_{\eps=0}(z)$ must be located in $\D$, since all unimodular eigenvalues of $\AA_0$ are $\{ \omega_1, \ldots, \omega_M \}$. By continuity of zeros of $\chi_\eps(z)$ with respect to $\eps$, it follows that all the {zeros} of $\tilde{\chi}_\eps(z)$ also lie inside $\D$ for $|\eps| < \eps_0$ sufficiently small. Thus there exists some constant $\rho < 1$ as stated in (iii). 
\end{proof}

As a next step, we apply the results from Proposition \ref{prop:perturbation} to study the explicit formula as $t \to T$ corresponding to the limit $\eps \to 0$ above. In what follows, we always assume that $|\eps| < \eps_0$ with the constant $\eps_0>0$ from above. 

We define the projection
$$
P(\eps) := \sum_{j=1}^M P_j(\eps) :W \to W
$$
onto the subspace spanned by the vectors $f_1(\eps), \ldots, f_M (\eps) \in W$, where we note that $P_j(\eps)=P_j(\eps)^2$ and $P_j(\eps) P_k(\eps) = 0$ for $j \neq k$. Thus $P(\eps)$ is also a projection on $W$, {but it is not necessarily orthogonal when $\eps\neq0$}.

For a fixed vector $v \in W$, we define
\be \label{def:v}
v(\eps,z) := \left \langle (\id-z \AA(\eps))^{-1} P(\eps) v, 1 \right \rangle \quad \mbox{for $z \in \D$ and $|\eps| < \eps_0$}.
\ee
Since $\AA(\eps)$ simply acts as $\mathrm{diag}(\omega_1(\eps), \ldots, \omega_M(\eps))$ on $\ran \, P(\eps)$, we obtain
\be \label{eq:v_bubbles}
\boxed{v(\eps,z) = \sum_{j=1}^M \frac{\beta_j(\eps)}{1- \omega_j(\eps) z} \quad \mbox{with} \quad \beta_j(\eps) := \langle P_j(\eps) v, 1 \rangle}{.}
\ee
Since $|\omega_j(\eps)| < 1$ for $\eps \neq 0$, we see that $v(\eps,z)$ is a rational function of $z$ with poles in the complement of the closed unit disk $\ov{\D}$, and thus $v(\eps) \in H^\infty_+(\T)$ for $\eps \neq 0$. Also, we remark that $v(0,z) = 0$ for all $z \in \D$, since $\beta_j(0)= \langle P_j(0)v, 1 \rangle = 0$ because $f_j \perp 1$ for all $j=1, \ldots, M$.

 The next result explicitly determines the blow-up behavior of $\| v(\eps) \|_{H^s}$ as $\eps \to 0$.

\begin{lemma}[Blow-up rates] \label{lem:blowup_unitarypart}
Let $v \in W$ be given and consider $v(\eps)$ as given by \eqref{def:v}. For all $1 \leq j \leq M$, we have
$$
|\beta_j(\eps)|^2 = \left ( |\langle v, f_j \rangle |^2 + \O(\eps) \right ) \cdot \left (1-|\omega_j(\eps)|^2 \right )
$$
$$
 1-|\omega_j(\eps)|^2 = c_j \eps^{2 \nu_j} + \O(\eps^{2 \nu_j + 1} ),
$$
with some constants $c_j > 0$ and integers $\nu_j \geq 1$. 

Furthermore, it holds that
$$
{\| v(\eps) \|_{H^s}^2 \sim_{s,u_0,v} \sum_{j=1}^M |\langle v, f_j \rangle |^2 |\eps|^{-4s \nu_j} \quad \mbox{for any $s > 0$ as $\eps \to 0$}},
$$
provided that $\langle v, f_j \rangle \neq 0$ for $j=1, \ldots, M$.
\end{lemma}

\begin{remark*}
{\em Below we will take $v=u_0$ and the non-vanishing conditions $\langle v, f_j \rangle \neq 0$ for $j=1, \ldots, M$ will be automatically satisfied thanks to Lemma \ref{lem:Wu}.}
\end{remark*}
\begin{proof}
We divide the proof into the following steps.

\medskip
{\bf Step 1.} Fix $1 \leq j \leq M$. We first claim that
\be \label{eq:defect_f_j}
|\langle f_j(\eps), 1 \rangle |^2 = 1- |\omega_j(\eps)|^2.
\ee
Indeed, we note that $\AA(\eps) f_j(\eps) = {\omega_j(\eps) f_j(\eps)}$ means that $\eu^{-it \LL} S^* f_j(\eps) = \omega_j(\eps) f_j(\eps)$ with $t=T-\eps$. Since $\eu^{-it \LL}$ is unitary and $\|f_j(\eps) \|_{L^2}^2 = 1$, we deduce
$$
\| \AA(\eps) f_j(\eps) \|_{L^2}^2 = \| S^* f_j(\eps) \|_{L^2}^2 = 1- |\langle f_j(\eps), 1 \rangle |^2{.}
$$
On the other hand, taking the $L^2$-norm in the eigenvalue equation {yields}
$$
{\| \AA(\eps) f_j(\eps ) \|_{L^2}^2 = |\omega_j(\eps)|^2}{,}
$$
{since} $\| f_j(\eps) \|_{L^2}^2 = 1$, which proves \eqref{eq:defect_f_j}. 

Next, we recall that $P_j(\eps) v = \langle v, g_j(\eps) \rangle f_j(\eps)$. Since $g_j(\eps) = f_j + \O(\eps)$ and $\| f_j(\eps) \|_{L^2}^2 = 1$, we can use \eqref{eq:defect_f_j} to deduce
$$
|\beta_j(\eps)|^2 = |\langle  v, g_j(\eps) \rangle|^2 \cdot |\langle {f}_j(\eps), 1 \rangle|^2 = \left ( |\langle v, f_j \rangle|^2 + \O(\eps) \right ) \cdot (1- |\omega_j(\eps)|^2 ).
$$

Next, we consider the analytic function $\rho_j(\eps) := 1 - |\omega_j(\eps)|^2 \geq 0$. Since $\rho_j(\eps) > 0$ for $\eps \neq 0$ and $\rho_j(0)=0$, we infer
$$
1-|\omega_j(\eps)|^2 = \rho_j(\eps)= c_j \eps^{2\nu_j} + \O(\eps^{2\nu_j+1})
$$ 
with some constant $c_j > 0$ and some integer $\nu_j \geq 1$.

\medskip
{{\bf Step 2.} Let $s>0$ be given and take $\eps\in(-\eps_0,\eps_0)$ with $\eps\neq0$.} We study the blow-up behavior of $\| v(\eps) \|_{H^s}$ with
$$
v(\eps, z) = \sum_{j=1}^M v_j(\eps, z) = \sum_{j=1}^M \frac{\beta_j(\eps)}{1 - \omega_j(\eps) z} \quad \mbox{as} \quad \eps \to 0.
$$
For the 0-th Fourier mode, we find
 $$
{\widehat{v(\eps)}(0) = v(\eps, 0) = \sum_{j=1}^M \beta_j(\eps) \to 0 \quad \mbox{as $\eps \to 0$}.}
 $$
By the equivalence $\| v(\eps) \|_{H^s}^2 \sim {|\widehat{v(\eps)}(0)|^2} + \| v(\eps) \|_{\dot{H}^s}^2$, it suffices to consider the homogeneous $\dot{H}^s$-norm, i.e.,
$$
 \| v(\eps) \|_{\dot{H}^s}^2  = \big \langle \sum_{j=1}^M v_j(\eps), \sum_{\ell=1}^M v_\ell(\eps) \big \rangle_{\dot{H}^s} = \sum_{j=1}^M \| v_j(\eps) \|_{\dot{H}^s}^2 + {\sum_{\substack{1\leq j,\ell\leq M\\j\neq\ell}} \langle v_j(\eps), v_\ell(\eps) \rangle_{\dot{H}^s}}{.}
 $$
 Here, for elements $g(z) = \sum_{n\geq 0} {g_n} z^n$ and $h(z) = \sum_{n \geq 0} h_n z^n$ in $H^s_+(\T)$, we have $\langle  g,h \rangle_{\dot{H}^s} = \sum_{n \geq 1} n^{2s} g_n \bar{h}_n$. Now, by the geometric series,
$$
v_j(\eps,z) = \frac{\beta_j(\eps)}{1-\omega_j(\eps) z} = \beta_j(\eps) \sum_{n \geq 0} \omega_j(\eps)^n z^n \quad \mbox{for $z \in \D$},
$$
whence it follows that
$$
\| v_j(\eps) \|_{\dot{H}^s}^2 = |\beta_j(\eps)|^2 \sum_{n = 1}^\infty n^{2s} |\omega_j(\eps)|^{2n} = |\beta_j(\eps)|^2 \cdot \mathrm{Li}_{-2s}(|\omega_j(\eps)|^2)
$$
where $\mathrm{Li}_{\alpha}(\cdot)$ denotes the polylogarithm of order $\alpha$. Since $|\omega_j(\eps)| \to 1$ as $\eps \to 0$, a well-known asymptotic estimate for $\mathrm{Li}_{-2s}(q)$ as $q \to 1$ yields
$$
\| v_j(\eps) \|_{\dot{H}^s}^2  = |\beta_j(\eps)|^2 \cdot \mathrm{Li}_{-2s}(|\omega_j(\eps)|^2) \sim \Gamma(2s+1) |\beta_j(\eps)|^2 \rho_j(\eps)^{-(2s+1)}
$$
with the functions $\rho_j(\eps) = 1-|\omega_j(\eps)|^2$ introduced above. Let us now assume that
$$
\langle v,f_j \rangle \neq 0 \quad \mbox{for} \quad j=1, \ldots, M.
$$
This implies $|\beta_j(\eps)|^2 \sim |\langle v, f_j \rangle|^2 \rho_j(\eps)$ from the relation shown in Step 1 above, and we infer
$$
\| v_j (\eps) \|_{\dot{H}^s}^2 \sim_s |\langle v, f_j \rangle|^2 \rho_j(\eps)^{-2s} \quad \mbox{as $\eps \to 0$}{.}
$$
Since $\rho_j(\eps) = 1-|\omega_j(\eps)|^2 \sim \eps^{2\nu_j}$ as $\eps \to 0$ with some integer $\nu_j \geq 1$ by Step 1, we finally conclude
\be 
\| v_j(\eps) \|_{\dot{H}^s}^2 \sim_s |\langle v,f_j \rangle|^2 \cdot {|\eps|^{-4s\nu _j}} \quad \mbox{as} \quad \eps \to 0{.}
\ee

{
Next, we claim that the cross terms satisfy
\be \label{eq:cross}
\left|
\langle v_j(\eps),v_\ell(\eps)\rangle_{\dot H^s}
\right|
=
\O\bigl(|\eps|^{\nu_j+\nu_\ell}\bigr)
=
\O(|\eps|^2)
\quad \mbox{for $j\neq\ell$}.
\ee
Indeed, by the geometric-series representation above,
\begin{align*}
\langle v_j(\eps),v_\ell(\eps)\rangle_{\dot H^s}
&=
\beta_j(\eps)\overline{\beta_\ell(\eps)}
\sum_{n=1}^{\infty}
n^{2s}
\bigl(\omega_j(\eps)\overline{\omega_\ell(\eps)}\bigr)^n \\
&=
\beta_j(\eps)\overline{\beta_\ell(\eps)}
\operatorname{Li}_{-2s}
\bigl(\omega_j(\eps)\overline{\omega_\ell(\eps)}\bigr).
\end{align*}
Set
$$
q_{j\ell}(\eps)
:=
\omega_j(\eps)\overline{\omega_\ell(\eps)}.
$$
Since
$$
q_{j\ell}(\eps)
\longrightarrow
\omega_j\overline{\omega_\ell}\neq1
\quad \mbox{as $\eps\to0$}
$$
for $j\neq\ell$, after decreasing $\eps_0$ if necessary there exists $c>0$ such that
$$
|q_{j\ell}(\eps)-1|\geq c
$$
for all $|\eps|<\eps_0$ and all $j\neq\ell$. The function $\operatorname{Li}_{-2s}$, initially defined on $\D$, extends analytically across every compact subset of $\pt\D\setminus\{1\}$. Consequently,
$$
\sup_{\substack{|\eps|<\eps_0\\j\neq\ell}}
\left|
\operatorname{Li}_{-2s}(q_{j\ell}(\eps))
\right|
<\infty.
$$
Moreover, Step~1 gives $|\beta_j(\eps)|=\O(|\eps|^{\nu_j})$. Hence \eqref{eq:cross} follows.

Since there are only finitely many cross terms, we conclude that
\begin{align*}
\|v(\eps)\|_{\dot H^s}^2
&=
\sum_{j=1}^M
\|v_j(\eps)\|_{\dot H^s}^2
+\O(|\eps|^2) \\
&\sim_{s,u_0,v}
\sum_{j=1}^M
|\langle v,f_j\rangle|^2
|\eps|^{-4s\nu_j}
\quad \mbox{as $\eps\to0$},
\end{align*}
provided that $\langle v,f_j\rangle\neq0$ for every $j=1,\ldots,M$. Together with $\widehat{v(\eps)}(0)\to0$, this yields the asserted estimate for the $H^s$-norm and completes the proof of Lemma \ref{lem:blowup_unitarypart}.
}
\end{proof}

Next, we consider the projection complementary to $P(\eps)$, i.e., we set
$$
P_{\mathrm{rem}}(\eps) :={\id - P(\eps)} : W \to W{.}
$$
For a given vector $w \in W$, we define
$$
\boxed{w(\eps, z) = \langle ( \id - z \AA(\eps))^{-1} P_{\mathrm{rem}}(\eps) w, 1 \rangle }{.}
$$

\begin{lemma} \label{lem:control_nonunitarypart}
There exists some $R > 1$ such that the following holds. For all $|\eps| < \eps_0$, the functions $w(\eps,z)$ are rational in $z$ and holomorphic in $\D_R = \{ |z| < R \}$. Moreover, we have 
$$
\mbox{$w(\eps) \to w(0)$ in $H^\infty_+(\T)$ as $\eps \to 0$}.
$$
\end{lemma}

\begin{proof}
Assume that $|\eps| < \eps_0$ holds. Since the Riesz projections $P_j(\eps)$ commute with $\AA(\eps)$, so does the projection $P_{\mathrm{rem}}(\eps)$ commute with $\AA(\eps)$. {Hence $\ran P_{\mathrm{rem}}(\eps)$ is invariant under $\AA(\eps)$, and we define the following operator on $W$:}
$$
{\BB(\eps):=\AA(\eps)P_{\mathrm{rem}}(\eps):W\longrightarrow W.}
$$
Using that $P_{\mathrm{rem}}(\eps)^2 = P_{\mathrm{rem}}(\eps)$ and $\AA(\eps) P_{\mathrm{rem}}(\eps) = \BB(\eps)$, we find
$$
w(\eps,z) = \langle (\id - z \AA(\eps))^{-1} P_{\mathrm{rem}}(\eps) w, 1 \rangle = \langle (\id- z\BB(\eps))^{-1} P_{\mathrm{rem}}(\eps) w, 1 \rangle.
$$
{By expressing the inverse of a matrix through its adjugate,} we deduce
$$
w(\eps,z) = \frac{\langle \mathrm{adj} (\id-z \BB(\eps)) P_{\mathrm{rem}}(\eps) w,1 \rangle}{\det(\id- z \BB(\eps))} = \frac{p_\eps(z)}{q_\eps(z)}
$$
with some polynomials $p_\eps(z)$ and $q_\eps(z)$. By Proposition \ref{prop:perturbation} (iii), the spectrum $\sigma(\BB(\eps))$ belongs to $\{ |z| \leq \rho \}$  for all $|\eps| < \eps_0$ with some constant $0 < \rho < 1$. Thus the polynomials $q_\eps(z)$ have all their zeros in $\{ |z| \geq 1/\rho \}$. Thus by taking $1 < R < 1/\rho$, we deduce that $w(\eps, z)$ is rational in $z$ and holomorphic in $\D_R= \{ |z| < R \}$ for $|\eps| < \eps_0$.

It remains to show that
\be \label{eq:w_conv}
\mbox{$w(\eps) \to w(0)$ in $H^\infty_+(\T)$ as $\eps \to 0$}.
\ee
To see this, we first note that $\BB(\eps) = \BB(\eps) P_{\mathrm{rem}}(\eps)$ is continuous (in operator norm on the finite-dimensional space $W$). Thus, after possibly decreasing $\eps_0 > 0$ and $R > 1$, we obtain the uniform resolvent bound
\be \label{ineq:resolvent}
{C_R := \sup_{|\eps| < \eps_0} \sup_{|z| \leq R} \left \| (\id-z \BB(\eps))^{-1} \right \| < \infty.}
\ee
Using the resolvent identity
$$
(\id - z \BB(\eps))^{-1} - (\id - z \BB(0))^{-1} = (\id - z\BB(\eps))^{-1} z(\BB(\eps)-\BB(0)) (\id - z \BB(0))^{-1}
$$
together with \eqref{ineq:resolvent} and that $\BB(\eps) \to \BB(0)$ and $P_{\mathrm{rem}}(\eps) \to P_{\mathrm{rem}}(0)$ strongly, we obtain
\be \label{eq:uniform}
\sup_{|z| \leq R} |w(\eps,z)- w(0,z)| \to 0 \quad \mbox{as} \quad \eps \to 0.
\ee
In particular, $w(\eps,z) \to w(0,z)$ uniformly on every circle $|z|=r$ with $1 < r < R$.

{Fix such an $r$ with $1<r<R$.} Write the series expansions
$$
w(\eps,z) = \sum_{n \geq 0} c_n(\eps) z^n, \quad w(0,z) = \sum_{n \geq 0} c_n(0) z^n.
$$
By Cauchy's integral formula, 
\be \label{ineq:c_n}
c_n(\eps) - c_n(0) = \frac{1}{2 \pi i} \int_{|\zeta|=r} \frac{w(\eps,\zeta)-w(0, \zeta)}{\zeta^{n+1}} \, d \zeta.
\ee
Therefore, we obtain the estimate
$$
|c_n(\eps)-c_n(0)| \leq r^{-n} \delta_\eps \quad \mbox{with} \quad \delta_\eps:= \sup_{|\zeta|=r} |w(\eps,\zeta)-w(0,\zeta)|.
$$
Note that, by \eqref{eq:uniform}, we have $\delta_\eps \to 0$ as $\eps \to 0$. Let $s >0$ be given. From \eqref{ineq:c_n} we get
$$
\| w(\eps) - w(0) \|_{H^s}^2 = \sum_{n \geq 0} (1+n^2)^s |c_n(\eps)-c_n(0)|^2 \leq \delta_\eps^2 \sum_{n \geq 0} (1+n^2)^s r^{-2n}.
$$
Since $r > 1$, {the series on the right-hand side converges}. Hence 
$$
\|w(\eps)- w(0) \|_{H^s} \leq C_{s,r} \delta_\eps \to 0 \quad \mbox{as} \quad \eps \to 0.
$$
Because $s > 0$ is arbitrary, this proves \eqref{eq:w_conv}.

The proof of Lemma \ref{lem:control_nonunitarypart} is now complete.
\end{proof}

\subsection{Proof of Theorem \ref{thm:universal}}
{
Let $u_0\in H^\infty_+(\T)$ be a finite-gap potential. Suppose that the corresponding solution $u\in C([0,T);H^\infty_+(\T))$ of \eqref{eq:CS} with $u(0)=u_0$ blows up at some finite time $0<T<\infty$. Let
$$
W=(z\psi_{u_0}L^2_+)^\perp
$$
be the corresponding finite-dimensional model space. Recall that $u_0\in W$ and that $W$ is invariant under $L_{u_0}$ and $S^*$. Following the previous subsection, we consider
$$
\AA(\eps)
=
\eu^{-i(T-\eps)\LL}S^*|_W
=
\eu^{i\eps\Ls}\AA_0,
\qquad
\AA_0=A_T,
\qquad
\eps\in\R,
$$
where $\LL=2L_{u_0}+\id$ and $\Ls=\LL|_W$. By the blow-up criterion, the contraction
$$
\AA_0=A_T=\eu^{-iT\LL}S^*|_W
$$
has a non-trivial unitary subspace. By Lemmas \ref{lem:simple} and \ref{lem:no_total_loss}, we have
$$
W_{\mathrm{u}}
=
\operatorname{span}\{f_1,\ldots,f_M\},
\qquad
1\leq M < \|u_0\|_{L^2}^2,
$$
where $f_1,\ldots,f_M$ form an orthonormal set of eigenvectors of $\AA_0$ corresponding to the simple unimodular eigenvalues $\omega_1,\ldots,\omega_M$.
}

Next, we suppose $\eps > 0$ with $\eps < \eps_0$, where $\eps_0 > 0$ is the small constant given by Proposition \ref{prop:perturbation}. Using the explicit formula for $u(t)=\Us(t) u_0$ with $t= T-\eps$ and $u_0 \in W$, we can write
$$
\boxed{u(T-\eps) =\Us(T-\eps) P(\eps) u_0 + \Us(T-\eps) P_{\mathrm{rem}}(\eps)u_0 \quad \mbox{for $\eps \in (0, \eps_0)$}}{.}
$$ 
Here we recall that
$$
P(\eps) = \sum_{j=1}^M P_j(\eps) \quad \mbox{and} \quad {P_{\mathrm{rem}}(\eps) = \id - P(\eps)}
$$
denote the projections introduced in the previous subsection. Note that, by the definition of $\AA(\eps)$, the explicit formula reduces to $W$, i.e.,
$$
(\Us(T-\eps)f)(z) = \langle (\id - z \AA(\eps))^{-1} f, 1 \rangle \quad \mbox{for any $f \in W$ and $z \in \D$}.
$$
We now split the discussion for $f = P(\eps) u_0 \in W$ and $f = P_{\mathrm{rem}}(\eps) u_0 \in W$ into the following steps.

\medskip
{\bf Step 1.} From \eqref{def:v} and \eqref{eq:v_bubbles} together with Lemma \ref{lem:blowup_unitarypart}, we deduce
$$
v(\eps,z):=(\Us(T-\eps) P(\eps) u_0)(z) = \sum_{j=1}^M \frac{\beta_j(\eps)}{1-\omega_j(\eps) z} \quad \mbox{for $z \in \D$ and $\eps \in (-\eps_0, \eps_0)$},
$$
with the expansions
$$
|\beta_j(\eps)|^2 = \left ( |\langle u_0, f_j \rangle |^2 + \O(\eps) \right ) \cdot \left (1-|\omega_j(\eps)|^2 \right ), 
$$
$$
 |\omega_j(\eps)|^2 = 1- c_j \eps^{2 \nu_j} + \O(\eps^{2\nu_j + 1} ),
$$
with some constants $c_j > 0$ and integers $\nu_j \geq 1$. By Lemma \ref{lem:Wu}, we have 
\be \label{eq:Wu_simple}
|\langle u_0, f_j\rangle | = 1 \quad \mbox{for} \quad 1 \leq j \leq M. 
\ee
{Upon setting $\beta_j(t):=\beta_j(T-t)$ and $\omega_j(t):=\omega_j(T-t)$, we obtain analytic functions on $(T-\eps_0,T+\eps_0)$ with the expansions stated in Theorem \ref{thm:universal}.} Furthermore, by Lemma \ref{lem:blowup_unitarypart} and \eqref{eq:Wu_simple}, we obtain the precise blow-up rate
$$
\|v (\eps) \|_{H^s}^2 {\sim_{s,u_0}} \sum_{j=1}^M \eps^{-4s\nu_j} {\sim_{s,u_0}} \eps^{-4s\nu}{\quad \mbox{as $\eps\to0^+$}}
$$
for any $s > 0$ with the integer $\nu:= \max_{1 \leq j \leq M} \nu_j$.

\medskip
{\bf Step 2.} We next consider
$$
w(\eps,z) := (\Us(T-\eps) P_{\mathrm{rem}}(\eps) u_0)(z)= \langle (\id - z \AA(\eps))^{-1} P_{\mathrm{rem}}(\eps) u_0, 1 \rangle{.}
$$
From Lemma \ref{lem:control_nonunitarypart} we deduce that 
$$
w(\eps) \to w(0) \quad \mbox{in $H^\infty_+(\T)$ as $\eps \to 0$}{.}
$$
As a consequence, {setting $\eps=T-t$,} we find that the solution $u \in C([0,T); H^\infty_+(\T))$ can be written as
$$
u(t,z) = v(T-t, z) + w(T-t,z) = \sum_{j=1}^M \frac{\beta_j(t)}{1-\omega_j(t) z} + w(T-t, z)
$$
with the functions $\beta_1(t), \ldots, \beta_M(t)$ and $\omega_1(t), \ldots, \omega_M(t)$ as stated in Theorem \ref{thm:universal}, where 
$$
w(T-t) \to u_* := w(0) \quad \mbox{in $H^\infty_+(\T)$ as $t \nearrow T$},
$$
$$
{\|u(t)\|_{H^s}=(1+o(1))\|v(T-t)\|_{H^s}\sim_{s,u_0}(T-t)^{-2s\nu}\quad \mbox{for any $s>0$ as $t\nearrow T$},}
$$
with some integer $\nu \geq 1$. 

Next, we claim that the loss of $L^2$-mass at blow-up is equal to the integer $M \geq 1$, i.e., 
\be \label{eq:L2_quant}
 \| u_* \|_{L^2}^2 = \| u_0 \|_{L^2}^2 - M.
\ee
To see this, we observe that
$$
{u_*(z) = \langle (\id-z \AA_0)^{-1} P_{\mathrm{rem}}(0) u_0, 1 \rangle \quad \mbox{for $z \in \D$}.}
$$
Since $\AA_0^n f \to 0$ in $L^2_+$ for any $f \in \ran \, P_{\mathrm{rem}}(0)$ (because the spectrum of $\AA_0 |_{\ran \, P_{\mathrm{rem}}(0)}$ lies inside $\D$), we deduce that 
$$
\|u_* \|_{L^2}^2 = \| P_{\mathrm{rem}}(0) u_0 \|_{L^2}^2 - \lim_{n \to \infty} \| \AA_0^n P_{\mathrm{rem}}(0) u_0 \|_{L^2}^2 = \| P_{\mathrm{rem}}(0) u_0 \|_{L^2}^2.
$$
{On the other hand, $P(0)=\id-P_{\mathrm{rem}}(0)$ is the orthogonal projection onto the unitary part of the contraction $\AA_0$. Hence}
\be \label{eq:pythagoras}
\|u_0\|_{L^2}^2 = \| P(0) u_0 \|_{L^2}^2 +  \| u_* \|_{L^2}^2.
\ee
But by \eqref{eq:Wu_simple} we deduce
\be \label{eq:pythagoras2}
{\|P(0)u_0\|_{L^2}^2=\sum_{j=1}^M|\langle u_0,f_j\rangle|^2=M.}
\ee
This shows \eqref{eq:L2_quant} and completes the proof of (i) in Theorem \ref{thm:universal}.

\medskip
\textbf{Step 3.} It remains to prove items (ii) and (iii). 

{To prove (ii), recall that $\lim_{t\nearrow T}\omega_j(t)=\omega_j$ for $j=1,\ldots,M$. The numbers $\omega_1,\ldots,\omega_M\in\pt\D$ are the unimodular eigenvalues of $\AA_0$. By Lemma \ref{lem:simple}, they satisfy $\omega_j\neq\omega_\ell$ whenever $j\neq\ell$.}

To show (iii), we argue as follows. Recall from \eqref{eq:pythagoras} and \eqref{eq:pythagoras2} that
$$
\| u_0 \|_{L^2}^2 = \| P(0) u_0 \|_{L^2}^2 + \| u_* \|_{L^2}^2  = M + \| u_* \|_{L^2}^2
$$
where $P(0)$ denotes the orthogonal projection {of} $W$ onto the unitary subspace $W_{\mathrm{u}}$ of $A_T$. From Lemma \ref{lem:no_total_loss} we recall the strict inequality $M < \| u_0 \|_{L^2}^2$, which implies that $u_* \neq 0$.

The proof of Theorem \ref{thm:universal} is now complete. \hfill $\qed$

\section{Existence of finite-time blow-up}

In this section, we construct an explicit family of finite-gap potentials $u_0\in H^\infty_+(\T)$ leading to finite-time {blow-up} of the corresponding solution of \eqref{eq:CS}. For this, we consider finite-gap potentials of the form
\be  \label{eq:finite_gap_blowup}
\boxed{u_0(z) = z^{m} \beta_p(z) \left (a + \frac{c}{1- \bar{p} z} \right ) \quad \mbox{with} \quad \beta_p(z) := \frac{z-p}{1- \bar{p} z}}
\ee
with some integer $m \geq 0$, $p \in \D \setminus \{0 \}$ and $a,c \in \C$ satisfying the constraint
\be \label{eq:constraints22}
\bar{a} c+ \frac{|c|^2}{1-|p|^2} = 2
\ee
with $a\neq 0$ if $m \neq 0$. From the discussion in Section \ref{sec:finite_gap}, we have that $L_{u_0}$ and $S^*$ act invariantly on the finite-dimensional subspace
$$
W= (z \psi_{u_0} L^2_+)^\perp  \quad \mbox{with} \quad \dim W = m + 3
$$
and the inner function $\psi_{u_0}(z) = z^m \beta_p(z)^2$. Since $u_0, 1 \in W$, the explicit formula for $u(t,z)$ reduces to $W$ with the family of contractions 
$$
A_t = \eu^{- it \LL} S^* |_W \quad \mbox{with} \quad t \in \R
$$
{where} $\LL = 2 L_{u_0} + \id$. In the next subsection, we will see that the existence of unimodular eigenvalues for $A_{t_*}$ for some $t_* \in \R$ (equivalent to finite-time blow-up by Lemma \ref{lem:blow_up_criterion}) can be identified with the `resonance' condition
$$
2a+c = 0.
$$

\subsection{Existence of unimodular eigenvalues for $A_t$ when $m=0$}
We consider the case $m=0$ in \eqref{eq:finite_gap_blowup}. Thus, we consider the finite-gap potentials
\be \label{eq:finite_gap_m0}
u_0(z) = \beta_p(z) \left (a + \frac{c}{1- \bar{p} z} \right ) \quad \mbox{with} \quad \beta_p(z)  = \frac{z-p}{1- \bar{p} z}
\ee
with some $p \in \D \setminus \{0 \}$ and $a,c \in \C$ satisfying \eqref{eq:constraints22}.  In this case, we obtain from the discussion in Section \ref{sec:finite_gap} the orthogonal decomposition 
$$
W = W_0 \oplus \C \psi_{u_0} \quad \mbox{with} \quad W_0 = (\psi_{u_0} L^2_+)^\perp
$$
with the Blaschke product $\psi_{u_0} = \beta_p^2$ of degree 2. By well-known results for model spaces, see, e.g., \cite{GaMaRo}, we find
$$
W_0 = \mathrm{span} \left \{ e_0, e_1 \right \} \quad \mbox{with} \quad e_0(z) := \frac{1}{1-\bar{p}z}, \quad e_1(z) := \frac{z}{(1-\bar{p} z)^2}.
$$
Note that $(e_0, e_1)$ is {\em not} an orthonormal basis for $W_0$, but it turns out {to be a convenient choice} for the calculations below. {Set $e_2(z) := \psi_{u_0}(z)$. Then} $(e_0, e_1, e_2)$ is a basis for $W = W_0 \oplus \C \psi_{u_0}$ with $e_2 \perp e_j$ for $j=0,1$. Furthermore, we write\footnote{We use the letter $\Bs$, since $S^*$ denotes the backward shift.}
$$
\Ls = [L_{u_0}|_W]_{(e_0, e_1, e_2)} \quad \mbox{and} \quad \Bs = [S^* |_W]_{(e_0,e_1, e_2)} 
$$
to denote the $3 \times 3$-matrices for $L_{u_0}|_W$ and $S^* |_W$ with respect to the basis $(e_0,e_1,e_2)$. 

\begin{proposition} \label{prop:LB_matrix}
For $u_0(z)$ as in \eqref{eq:finite_gap_m0} with $m=0$, we have
$$
\Ls = \begin{bmatrix} 0 & \alpha & 0 \\ \bar{p} & \beta & 0 \\ 0 & 0 & \nu_{u_0} \end{bmatrix} \quad \mbox{and} \quad \Bs = \begin{bmatrix}  \bar{p} & 1 & * \\ 0 & \bar{p} & * \\ 0 & 0 & 0 \end{bmatrix},
$$
with some $\nu_{u_0} \in \R$ and
$$
\alpha = \frac{2p(a+c)}{c(1-r)}, \quad \beta = \frac{2r(a+c)-2a-c(1-r)}{c(1-r)} \quad \mbox{with} \quad r = |p|^2.
$$
\end{proposition}

\begin{remarks*}
{\em 1) The block decomposition of $\Ls$ and $\Bs$ reflects the known fact that the two-dimensional subspace $W_0= \mathrm{span} \{ e_0, e_1 \} \subset W$ is invariant under $L_{u_0}$ and $S^*$.

2) Below we will show that unimodular eigenvalues for $A_t$ can occur if and only if the resonant condition $2a + c=0$ holds. In this case, the expressions for $\alpha$ and $\beta$ simplify to
$$
\alpha = \frac{p}{1-r}, \quad \beta = \frac{2r}{1-r} \quad \mbox{with} \quad r=|p|^2.
$$}
\end{remarks*}

\begin{proof}
The proof follows from direct calculations. We sketch the main steps below.

\medskip
{\bf Step 1.} Since $e_0(z) = \sum_{n \geq 0} \bar{p}^n z^n$ and $e_1(z) = \sum_{n \geq 0} n \bar{p}^{n-1} z^n$, we find
$$
S^* e_0 = \sum_{n \geq 0} \bar{p}^{n+1} z^n = \bar{p} e_0, \quad S^* e_1 = \sum_{n \geq 0} (n+1) \bar{p}^n z^n = e_0 + \bar{p} e_1.
$$
This directly shows that $S^*(W_0) \subset W_0$ with $W_0 = \mathrm{span} \{ e_0, e_1 \}$, and we obtain the first two {columns} of the matrix $\Bs$ as stated above. Moreover, we note
$$
\langle  e_2, S^* e_2 \rangle = \langle S e_2, e_2 \rangle = \langle z \beta_p^2, \beta_p^2 \rangle = \langle z, |\beta_p|^4 \rangle = \langle z,1 \rangle = 0.
$$
Since $e_2 \perp \{ e_0, e_1\}$, this shows that the matrix element $\Bs_{33}=0$ vanishes.

\medskip
{\bf Step 2.} From Section \ref{sec:finite_gap}, we recall that  $\psi_{u_0}=e_2$ is an eigenfunction of $L_{u_0}$, i.e., we have 
$$
L_{u_0} e_2 = \nu_{u_0} e_2
$$
with some $\nu_{u_0} \in \R$. Thus it {remains to determine} the upper-left $2\times 2$-block of $\Ls$ with respect to the basis $(e_0,e_1)$ of $W_0$.

Since $e_0(z)$ is the reproducing kernel at $p$, we have $f(p) = \langle f, e_0 \rangle$ for all $f \in L^2_+(\T)$. Because of ${u_0(p)=0}$, this implies $\langle f, {T_{\bar{u}_0}} e_0 \rangle = \langle {u_0f}, e_0 \rangle = {u_0(p)} f(p) = 0$ for any $f \in L^2_+(\T)$. This shows that ${T_{\bar{u}_0}} e_0 = 0$. Thus we find
$$
{L_{u_0}} e_0 = D e_0 = z \pt_z \left ( \frac{1}{1-\bar{p} z} \right ) = \bar{p} e_1.
$$
Next we compute ${T_{\bar{u}_0}} e_1$. Here we use that $f'(p) = \langle f, e_1 \rangle$ for all $f \in L^2_+(\T)$, since $e_1(z)$ is {the} derivative reproducing kernel at $p$. Using that ${u_0(p)=0}$ and $f(p) = \langle f, e_0 \rangle$, we infer
$$
\langle f, {T_{\bar{u}_0}} e_1 \rangle = \langle {u_0f}, e_1 \rangle = {(u_0f)'(p)} = {u_0'(p)} f(p) = {u_0'(p)} \langle f, e_0 \rangle \quad \mbox{for all $f \in L^2_+(\T)$}.
$$ 
Hence ${T_{\bar{u}_0}} e_1$ is parallel to $e_0$ with ${T_{\bar{u}_0}} e_1 = {\overline{u_0'(p)}} e_0$. We calculate
$$
{\overline{u_0'(p)}} = \frac{\bar{\xi}}{1-r} \quad \mbox{with $\displaystyle \xi = a + \frac{c}{1-r}$ and $r = |p|^2$}.
$$
Noting that \eqref{eq:constraints22} is equivalent to $\bar{\xi} c = 2$, we conclude
$$
{T_{u_0} T_{\bar{u}_0} e_1 = \frac{2}{c(1-r)} u_0 e_0.}
$$
Next, we observe that
$$
D e_1 = z \pt_z \left ( \frac{z}{s^2} \right ) = \frac{z}{s^2} + \frac{2 \bar{p} z^2}{s^3} \quad \mbox{with} \quad s = 1- \bar{p} z.
$$
Some straightforward manipulations yield
$$
{L_{u_0}} e_1 = \alpha e_0 + \beta e_1 + \frac{\gamma}{s^3} = \alpha e_0 + \beta e_1, 
$$
where the vanishing of $\gamma=0$ follows from using the constraint \eqref{eq:constraints22}, and we obtain
$$
\alpha = \frac{2 p (a+c)}{c (1-r)}, \quad \beta = \frac{2 r(a+c) -2a - c(1-r)}{c(1-r)}, \quad r = |p|^2.
$$
This shows that the upper-left $2 \times 2$-block of the matrix $\Ls$ is given as claimed. This completes the proof of Proposition \ref{prop:LB_matrix}.
\end{proof}

As a next step, we determine a necessary and sufficient condition {under which}
$$
A_t = \eu^{-it \LL} S^* |_W = \eu^{-it} {\eu^{-2itL_{u_0}}} S^* |_W
$$
has a unimodular eigenvalue for some $t_* \in \R$, i.e., we have $\sigma(A_{t_*}) \cap \pt \D \neq \emptyset$. We have the following result.

\begin{proposition} \label{prop:A_t_unimodular}
Let $u_0(z)$ be given as in \eqref{eq:finite_gap_m0} with $m=0$. 
\begin{enumerate}
\item[(i)] If $2a +c =0$, then $A_t$ has a unimodular eigenvalue if and only if
$$
t =( 2 \ell + 1)  \cdot \frac{  \pi (1-|p|^2)}{4 |p|} \quad \mbox{for some $\ell \in \Z$}.
$$
\item[(ii)] If $2a + c \neq 0$, then $A_t$ has no unimodular eigenvalue {for any $t \in \R$}.
\end{enumerate}
\end{proposition}

\begin{proof}
Since $A_t = \eu^{-it \LL} S^* |_W = {\eu^{-it}} \eu^{-2it L_{u_0}} S^* |_W$, we can equivalently discuss the unimodular eigenvalues of 
$$
\tilde{A}_\tau := \eu^{-i \tau L_{u_0}} S^* |_W \quad \mbox{for} \quad \tau \in \R
$$
via {the} simple relation $\tau = 2t$ and multiplication {by} the trivial phase factor ${\eu^{-it}}$. We organize the rest of the proof into the following steps.

\medskip
\textbf{Step 1.}  By the block structure of the matrices $\Ls$ and $\Bs$ in Proposition \ref{prop:LB_matrix}, we find that the matrix $\Ms_\tau \in \C^{3 \times 3}$ for $\tilde{A}_\tau$ with respect to the basis $(e_0,e_1, e_2)$ is given by the block form
$$
\Ms_\tau = \begin{bmatrix} \eu^{-i \tau \Ls_2 }\Bs_2 & * \\ 0 & 0 \end{bmatrix} \quad \mbox{with} \quad \Ls_2 := \begin{bmatrix} 0 & \alpha \\ \bar{p} & \beta \end{bmatrix} \quad \mbox{and} \quad  \Bs_2 := \begin{bmatrix} \bar{p} & 1 \\ 0 & \bar{p} \end{bmatrix}.
$$
Thus the spectrum of $\Ms_\tau$ decomposes as $\sigma(\Ms_\tau) = \sigma(\Ms_{2,\tau}) \cup \{ 0 \}$ and it remains to look for unimodular eigenvalues of the $2 \times 2$-matrix
$$
\Ms_{2,\tau} := \eu^{-i \tau \Ls_2} \Bs_2.
$$
We consider the characteristic polynomial 
$$
P_\tau(z) = \det(z \mathds{1}_2 - \Ms_{2,\tau}) = z^2 - \tr(\Ms_{2,\tau}) z + \det(\Ms_{2,\tau}).
$$
From the constraint equation \eqref{eq:constraints22} we see that $c \neq 0$. Thus we can define the real number
$$
      q:=\frac ac, 
$$
where $q \in \R$ follows from \eqref{eq:constraints22}. Note that
$$
        \beta=-2q+\frac{3r-1}{1-r},
$$
which shows that $\beta$ is a real number, too. Now, let $\lambda_+ \geq \lambda_-$ denote the two eigenvalues of $\Ls_2$, which must be real by {self-adjointness} of $L_{u_0}$, and let us set
$$
        \delta:=\lambda_+-\lambda_-.
$$
From the characteristic polynomial of $\Ls_2$ one obtains
$$
        \delta^2
        =\beta^2+4\alpha\overline p
        =(2q+1)^2+\frac{4r}{(1-r)^2}.
$$
In particular, we see that $\delta>0$.

Next, we observe that
$$
\det (\Ms_{2, \tau}) = \det(\eu^{-i\tau \Ls_2}) \det \Bs_2 = \eu^{-i \tau \tr \, \Ls_2} \bar{p}^2 = \eu^{-i \tau \beta} \bar{p}^2.
$$
To find an explicit expression for $\tr(\Ms_{2, \tau})$, we recall that the standard formula for the exponential of a $2\times2$ matrix gives
\[
\eu^{-i\tau \Ls_2}
=
\eu^{-i\tau\beta/2}
\left[
\cos\left(\frac{\tau\delta}{2}\right) \mathds{1}_2
-
\frac{2i}{\delta}\sin\left(\frac{\tau\delta}{2}\right)
\left(\Ls_2 -\frac{\beta}{2} \mathds{1}_2 \right)
\right]{,}
\]
{where we used again that $\tr (\Ls_2) = \beta$.} In view of  $\tr (\Bs_2)=2\overline p$, a direct computation gives
\[
        \tr\left(\left(\Ls_2-\frac{\beta}{2} \mathds{1}_2\right)\Bs_2\right)=\bar{p}.
\]
Consequently,
\[
\tr(\Ms_{2,\tau})
=
2\bar{p}\,\eu^{-i\tau\beta/2}
\left[
\cos\left(\frac{\tau\delta}{2}\right)
-
\frac{i}{\delta}\sin\left(\frac{\tau\delta}{2}\right)
\right].
\]
Thus we deduce that
\be \label{eq:P_tau}
\boxed{
P_\tau(z)
=
z^2
-
2\bar{p}\,\eu^{-i\tau\beta/2}
\left[
\cos\left(\frac{\tau\delta}{2}\right)
-
\frac{i}{\delta}\sin\left(\frac{\tau\delta}{2}\right)
\right]z
+
\bar{p}^{2}\eu^{-i\tau\beta}.
}
\ee

\medskip
\textbf{Step 2.} Let
\[
        \rho:=|p|, \qquad r=\rho^2,
        \qquad \theta:=\frac{\tau\delta}{2}.
\]
Set
\[
        z=\overline p\,\eu^{-i\tau\beta/2}w.
\]
Then $P_\tau(z)=0$ is equivalent to
\[
        w^2-2C_\tau w+1=0 \quad \mbox{with} \quad        C_\tau :=
        \cos\theta-\frac{i}{\delta}\sin\theta.
\]
Since $|z|=\rho |w|$, the polynomial $P_\tau$ has a unimodular zero if and only if the quadratic equation for $w$ has a root satisfying
\[
        |w|=\frac1\rho.
\]
The roots of $w^2-2C_\tau w+1$ have product $1$. Thus, if one root has modulus $1/\rho$, the other has modulus $\rho$. Writing the root of modulus $1/\rho$ as
\[
        w=\frac1\rho\eu^{i\varphi},
\]
we get
\[
        2C_\tau=w+\frac1w
        =\left(\frac1\rho+\rho\right)\cos\varphi
        +i\left(\frac1\rho-\rho\right)\sin\varphi.
\]
Therefore
\[
        \frac{4r}{(1+r)^2}(\mathrm{Re} \, C_\tau)^2
        +
        \frac{4r}{(1-r)^2}(\mathrm{Im} \, C_\tau)^2
        =1.
\]
Using $\mathrm{Re} \, C_\tau=\cos\theta$ and  $\mathrm{Im} \, C_\tau=-\frac1\delta\sin\theta$,  we obtain the equation
\be \label{eq:secular1}
        \frac{4r}{(1+r)^2}\cos^2\theta
        +
        \frac{4r}{\delta^2(1-r)^2}\sin^2\theta
        =1.
\ee
Recalling that $\theta = \tau \delta/2$, we conclude that
\be \label{eq:secular2}
\boxed{
        \frac{4r}{(1+r)^2}
        \cos^2\left(\frac{\tau\delta}{2}\right)
        +
        \frac{4r}{\delta^2(1-r)^2}
        \sin^2\left(\frac{\tau\delta}{2}\right)
        =1
}
\ee
holds if and only if $P_\tau(z)=0$ has a zero on $\pt \D$.

\medskip
\textbf{Step 3.} Recall the square identity
\[
        \delta^2=(2q+1)^2+\frac{4r}{(1-r)^2}.
\]
Hence
\[
        \frac{4r}{\delta^2(1-r)^2}\leq 1,
\]
with equality if and only if $2q+1=0$, i.e.
\[
        q=-\frac12.
\]
On the other hand,
\[
        \frac{4r}{(1+r)^2}<1,
\]
because $0<r<1$. Therefore, we deduce that \eqref{eq:secular2} can hold only if $q=-1/2$, i.e.,
\[
         2a+c = 0.
\]
Therefore, we deduce that, if $2a+c \neq 0$, then ${\tilde{A}_\tau}$ has no unimodular eigenvalue {for any $\tau\in \R$}. This completes the proof of statement (ii) in Proposition \ref{prop:A_t_unimodular}.

As for (i), let us now assume that 
$$
2a + c = 0,
$$
i.e., we have $q=-1/2$. In this case the second coefficient in \eqref{eq:secular2} is exactly $1$, and we see that $P_\tau(z)$ has a zero on $\pt \D$ if and only if
\[
        \cos^2\theta=0.
\]
Therefore
\[
        \theta=\frac\pi 2+\ell \pi  \quad \mbox{with $\ell \in \Z$}.
\]
Since $\theta=\tau\delta/2$, this gives
\[
        \frac{\tau\delta}{2}=\frac\pi2+\ell \pi,
\]
and hence
\[
        \tau =\frac{(2{\ell}+1)\pi}{\delta} \quad \mbox{with $\ell \in \Z$}.
\]
For $q=-1/2$, we readily find that
\[
        \delta=\lambda_+-\lambda_-=\frac{2|p|}{1-r}.
\]
Therefore, $\tilde{A}_{\tau}$ has {a} unimodular eigenvalue if and only if $\tau = \tau_\ell$ with
\[
        \tau_\ell =\frac{(2\ell+1)\pi(1-r)}{2|p|} \quad \mbox{with $\ell \in \Z$}.
\]
Recalling that $\tau = 2 t$ and $r=|p|^2$, we complete the proof of Proposition \ref{prop:A_t_unimodular}.
\end{proof}

Next, we let $\omega(t)$ denote the eigenvalue of $A_t$ above such that
$$
\omega(t) \to \omega \in \pt \D \quad \mbox{as} \quad t \to T := {t_0} = \frac{\pi(1-|p|^2)}{4|p|}.
$$
The next result shows that $|\omega(t)|^2$ approaches 1 quadratically as $t \to T$.

\begin{proposition} \label{prop:expansion_resonance}
For $A_t$ as in Proposition \ref{prop:A_t_unimodular} with $2a+c=0$, we have
$$
|\omega(t)|^2 = 1- \frac{4 r(1-r)}{(1+r)^3} (T-t)^2 + \O((T-t)^3),
$$
where $r = |p|^2$.
\end{proposition}

\begin{proof}
This follows from direct calculations and {an inspection of} the proof of Proposition \ref{prop:A_t_unimodular}. Indeed, we remark that the non-zero eigenvalues $\{ \lambda_+(t), \lambda_-(t) \}$ of $A_t$ can be calculated from the roots $\{ z_+(\tau), z_-(\tau) \}$ of the characteristic polynomial $P_\tau(z) = \det (z \mathds{1}_2 - \Ms_{2,\tau})$ via 
$$
\lambda_\pm(t) = {\eu^{-it}} z_\pm(2t).
$$ 
Let $z_+(\tau)$ be the zero of ${P_\tau(z)}$ with $|z_+(\tau)| \to 1$ as $\tau \to 2T$ and hence $\omega(t) = \lambda_+(t) = {\eu^{-it}} z_+(2 t)$. Denoting $s = |z_+(\tau)|^2$ and $x = \sqrt{r} \tau/(1-r)$, we infer from the proof of Proposition \ref{prop:A_t_unimodular} that $P_\tau(z)=0$ is equivalent to the quadratic equation
$$
\frac{4rs}{(s+r)^2} \cos^2 x + \frac{s(1-r)^2}{(s-r)^2} \sin^2 x = 1.
$$
Expanding near $s=1$ and $x = \pi/2$ yields that
$$
|z_+(\tau)|^2 = 1- \frac{r(1-r)}{(1+r)^3} (\tau - 2 T)^2 + \O((\tau-2 T)^3).
$$
Since $|\omega(t)| = |z_+(2t)|$, we obtain the claimed expansion.
\end{proof}

For later use, we also show {that} all eigenvalues of $A_t$ are uniformly bounded away from $\pt \D$ in the non-resonant case $2a+c \neq 0$.

\begin{proposition} \label{prop:control_non_resonance}
For $A_t$ as in Proposition \ref{prop:A_t_unimodular} with $2a+c \neq 0$, it holds
$$
r(A_t) \leq \rho_0 \quad \mbox{for all $t \in \R$}
$$
with some constant $\rho_0 < 1$ depending on {$u_0$}.
\end{proposition}

\begin{proof}
Let $\rho_0 := \sup_{t \in \R} r(A_t) \leq 1$. To show that $\rho_0 < 1$ holds, we argue by contradiction and assume that there exists a sequence $(t_n)$ in $\R$ such that $r(A_{t_n}) \to 1$. {{We remark that $|t_n| \to \infty$. Indeed, otherwise continuity of $r(A_t)$ would give $r(A_{t_*})=1$ at a finite accumulation point $t_*$, contrary to Proposition \ref{prop:A_t_unimodular}.}} Furthermore, from the proof of Proposition \ref{prop:A_t_unimodular}, we see that one of the roots of the characteristic polynomial $P_{\tau_n}(z)$ must approach $\pt \D$ along the sequence $\tau_n = 2 t_n$. From \eqref{eq:P_tau} and by passing to a subsequence if necessary, we deduce that
$$
P_{\tau_n}(z) \to {P_\infty(z) := z^2 -2 \bar{p} \zeta \left [ c_* - \frac{i}{\delta} s_* \right ] z + \bar{p}^2 \zeta^2}
$$
uniformly in $z \in \ov{\D}$ for some $\zeta \in \C$ with $|\zeta|=1$ and $c_*, s_* \in \R$ such that ${c_*^2 + s_*^2 = 1}$. {Choose roots $z_n$ of $P_{\tau_n}$ such that $|z_n|\to1$. Passing to a further subsequence, we have $z_n\to z_*$ for some $z_*\in\pt\D$. The uniform convergence of the characteristic polynomials yields $P_\infty(z_*)=0$.} As in the proof of Proposition \ref{prop:A_t_unimodular}, we analyze the roots of $P_\infty(z)$ by making the substitution $z = \bar{p} \zeta w$ to obtain the quadratic equation  given by
$$
w^2 - 2 C_* w + 1=0 \quad \mbox{with} \quad C_* := c_* - \frac{i}{\delta} s_*.
$$
By a straightforward adaptation of the arguments in the proof of Proposition \ref{prop:A_t_unimodular}, we deduce that $P_\infty(z)$ has a root on $\pt \D$ if and only if
\be \label{eq:ellipse}
\frac{4r}{(1+r)^2} c_*^2 + \frac{4r}{\delta^2 (1-r)^2} s_*^2 = 1
\ee
with the constant $\delta > 0$ as in the proof of Proposition \ref{prop:A_t_unimodular} and $r = |p|^2$. From $2a + c \neq 0$ we see that $q=\frac{a}{c} \neq -\frac{1}{2}$, whence it {follows that}
$$
\delta^2 = (2q+1)^2 + \frac{4r}{(1-r)^2} > \frac{4r}{(1-r)^2}.
$$
Since $\frac{4r}{\delta^2 (1-r)^2} < 1$, $\frac{4r}{(1+r)^2} < 1$ and  ${c_*^2 + s_*^2 = 1}$, we deduce that \eqref{eq:ellipse} cannot hold. This is the desired contradiction and we conclude that $\rho_0 < 1$ holds.
\end{proof}

\subsection{Proof of Theorem \ref{thm:blowup} (Single-bubble case)} Let us first consider finite-gap potentials $u_0 \in H^\infty_+(\T)$ given by \eqref{def:u0_blowup} with
$$
m=0.
$$
Recall that $W$ {denotes} the finite-dimensional subspace on which the explicit formula reduces. By Proposition \ref{prop:A_t_unimodular} (i), we see that $A_t = \eu^{-it \LL} S^* |_W$ has a unimodular eigenvalue if and only if  $t = t_\ell$ with
$$
t _\ell =  (2 \ell + 1)  \cdot \frac{  \pi (1-|p|^2)}{4 |p|} \quad \mbox{with $\ell \in \Z$}.
$$
Let $T = {t_0} = \frac{\pi(1-|p|^2)}{4 |p|} > 0$ be the {smallest such positive time}. By Lemma \ref{lem:blow_up_criterion}, we deduce that the corresponding solution $u \in C([0,T); H^\infty_+(\T))$ blows up at $T >0$. Thus we can apply Theorem \ref{thm:universal}, where we note that 
$$
M=1,
$$ 
since $1 \leq M \leq \|u_0 \|_{L^2}^2 < 2$ by \eqref{ineq:L2_mass_single}. To determine the blow-up rate with 
$$
\nu = \nu_1 = 1,
$$ 
we use the explicit expansion provided by Proposition \ref{prop:expansion_resonance}. This completes the proof of Theorem \ref{thm:blowup} when $m=0$.

It remains to discuss the case $m \geq 1$ for $u_0 \in H^\infty_+(\T)$ given by \eqref{def:u0_blowup}. Here we note that
$$
u_0(z) = z^m u_0^{(0)}(z),
$$
where $u_0^{(0)} \in H^\infty_+(\T)$ is a finite-gap potential given by \eqref{def:u0_blowup} with $m=0$. Thus the solution $u \in C(I; H^\infty_+(\T))$ of \eqref{eq:CS} with $u(0)={u_0}$ coincides with the one obtained by using the discrete {Galilean}-type symmetry in Lemma \ref{lem:Galilean}, i.e., we have $u(t) = \Gc_m u^{(0)}(t)$.  {Furthermore,} by Lemma \ref{lem:Galilean} (i), we deduce that $u \in C([0,T); H^\infty_+(\T))$ also blows up at time $T = \frac{\pi (1-|p|^2)}{4 |p|} > 0$. Thus we can apply Theorem \ref{thm:universal}, where we note that $M = 1$ since $1 \leq M < \| u_0 \|_{L^2}^2 < 2$ by \eqref{ineq:L2_mass_single} again. Finally, to see that
$$
\nu = \nu_1 = 1,
$$
we notice that Lemma \ref{lem:Galilean} (i) yields that
$$
\| u(t) \|_{H^s} \sim_{m,s} \| u^{(0)}(t) \|_{H^s} \sim_s  (T-t)^{-2s} \quad \mbox{as} \quad   t \nearrow T.
$$

{{The proof of Theorem \ref{thm:blowup} is now complete. \hfill $\qed$}}

\subsection{Proof of Theorem \ref{thm:blowup_multi} (Multi-bubble case)}
{
Let $u \in C([0,T);H^\infty_+(\T))$ denote the solution with initial datum $u_0$ appearing in Theorem \ref{thm:blowup}, and set
\[
T_k:=\frac{T}{k^2}.
\]
By Lemma \ref{lem:covering}, the function
\[
u^{(k)}(t,z)=\sqrt{k}\,u(k^2t,z^k)
\]
is a solution of \eqref{eq:CS} on $[0,T_k)$. Its initial datum is a finite-gap potential by Lemma \ref{lem:covering_finite_gap}; for $k=1$ this is simply the original datum. Moreover, for every $s>0$, Lemma \ref{lem:covering} gives
\[
\|u^{(k)}(t)\|_{H^s}
\geq
\sqrt{k}\,\|u(k^2t)\|_{H^s}.
\]
Consequently, $u^{(k)}$ blows up at $T_k$.

It remains to determine the blow-up dynamics. From Theorem \ref{thm:blowup}, we have
\[
u(\tau,z)-\frac{\beta(\tau)}{1-\omega(\tau)z}
\longrightarrow u_*(z)
\quad\mbox{in }H^\infty_+(\T)
\quad\mbox{as }\tau\nearrow T.
\]
Since $\omega(\tau)$ converges to a non-zero unimodular number, after possibly decreasing the neighborhood of $T$, we may choose an analytic function $\omega_1^{(k)}(t)$ such that
\[
\big(\omega_1^{(k)}(t)\big)^k=\omega(k^2t).
\]
For $j=1,\ldots,k$, define
\[
\omega_j^{(k)}(t)
:=
\omega_1^{(k)}(t)\eu^{2\pi i(j-1)/k},
\qquad
\beta_j^{(k)}(t)
:=
\frac{\beta(k^2t)}{\sqrt{k}}.
\]
The elementary partial-fraction identity
\[
\frac{1}{1-\omega(k^2t)z^k}
=
\frac1k\sum_{j=1}^k
\frac{1}{1-\omega_j^{(k)}(t)z}
\]
therefore yields
\[
u^{(k)}(t,z)
-
\sum_{j=1}^k
\frac{\beta_j^{(k)}(t)}{1-\omega_j^{(k)}(t)z}
\longrightarrow
u_*^{(k)}(z)
:=
\sqrt{k}\,u_*(z^k)
\]
in $H^\infty_+(\T)$ as $t\nearrow T_k$.

Writing
\[
|\omega(\tau)|^2
=
1-c(T-\tau)^2+\O((T-\tau)^3),
\]
we obtain, since $T-k^2t=k^2(T_k-t)$,
\[
1-\big|\omega_j^{(k)}(t)\big|^2
=
ck^3(T_k-t)^2+\O((T_k-t)^3)
\]
and
\[
\big|\beta_j^{(k)}(t)\big|^2
=
ck^3(T_k-t)^2+\O((T_k-t)^3).
\]
Thus $\nu_j=1$ for every $j=1,\ldots,k$. Furthermore, the limits
\[
\omega_{j,*}
=
\omega_{1,*}\eu^{2\pi i(j-1)/k},
\qquad
j=1,\ldots,k,
\]
are distinct and form a regular $k$-gon on $\pt\D$. Since these poles are distinct and their residues are non-zero for $t<T_k$ sufficiently close to $T_k$, the above decomposition contains exactly $k$ bubbles. Hence $M=k$.

Finally, the mass identity in Lemma \ref{lem:covering} and the corresponding identity in Theorem \ref{thm:blowup} give
\[
\|u_*^{(k)}\|_{L^2}^2
=
k\|u_*\|_{L^2}^2
=
k\big(\|u_0\|_{L^2}^2-1\big)
=
\|u_0^{(k)}\|_{L^2}^2-k.
\]
Also $u_*^{(k)}\neq0$, since $u_*\neq0$. This completes the proof of Theorem \ref{thm:blowup_multi}. \hfill $\qed$
}

\section{Global existence in the non-resonant case}

{
We first consider the case $m=0$. By Proposition \ref{prop:control_non_resonance}, there exists $\rho_0<1$ such that
\[
r(A_t)\leq \rho_0
\quad\mbox{for every }t\in\R.
\]
Hence Lemma \ref{lem:blow_up_criterion} shows that the corresponding solution exists globally in time.

We now prove the uniform Sobolev bounds. Since $W$ is finite-dimensional and the operators $A_t$ are contractions, the closure of the family
\[
\{A_t:t\in\R\}
\]
is compact in $\mathcal L(W)$. By continuity of the spectral radius, every operator in this closure has spectral radius at most $\rho_0$. Choose $R>1$ such that $R\rho_0<1$. It follows that $\id-zA_t$ is invertible for every $t\in\R$ and every $|z|\leq R$. Compactness therefore gives
\[
\sup_{t\in\R}\sup_{|z|\leq R}
\|(\id-zA_t)^{-1}\|_{\mathcal L(W)}
<+\infty.
\]
Using the finite-dimensional explicit formula, we conclude that
\[
\sup_{t\in\R}\sup_{|z|\leq R}|u(t,z)|
\lesssim_{u_0,R}1.
\]
By Cauchy's estimate, the Fourier coefficients satisfy
\[
|\widehat u(t,n)|
\lesssim_{u_0,R}R^{-n}
\quad\mbox{for every }n\geq0\mbox{ and }t\in\R.
\]
Consequently,
\[
\sup_{t\in\R}\|u(t)\|_{H^s}^2
\lesssim_{u_0,R}
\sum_{n\geq0}(1+n^2)^sR^{-2n}
<+\infty
\]
for every $s>0$.

It remains to consider $m\geq1$. Write
\[
u_0(z)=z^m v_0(z),
\qquad
v_0(z)=\beta_p(z)
\left(a+\frac{c}{1-\bar pz}\right).
\]
The datum $v_0$ satisfies the same constraint and the same non-resonance condition. Let $v$ be its global solution. By Lemma \ref{lem:Galilean}, the solution with initial datum $u_0$ is
\[
u(t,z)
=
\eu^{-im^2t}z^m v(t,\eu^{-2imt}z).
\]
It is therefore global, and Lemma \ref{lem:Galilean} also gives
\[
\sup_{t\in\R}\|u(t)\|_{H^s}
\lesssim_{m,s}
\sup_{t\in\R}\|v(t)\|_{H^s}
\lesssim_{u_0,s}1.
\]
This completes the proof of Theorem \ref{thm:gwp_finite_gap}. \hfill $\qed$

\medskip

To prove Corollary \ref{coro:instability}, let $u_0$ denote the one-pole datum underlying $u_0^{(k)}$ in Theorem \ref{thm:blowup_multi}. Thus its coefficients satisfy $c=-2a$. Choose a sequence
$q_n\in\R\setminus\{-1/2,0\}$ with $q_n\rightarrow-\frac12$ such that
\[
q_n+\frac{1}{1-|p|^2}>0,
\]
and define
\[
c_n
:=
\frac{c}{|c|}
\left(
\frac{2}{q_n+(1-|p|^2)^{-1}}
\right)^{1/2},
\qquad a_n:=q_nc_n.
\]
Then
\[
\bar a_nc_n+\frac{|c_n|^2}{1-|p|^2}=2,
\qquad
2a_n+c_n\neq0,
\]
and $(a_n,c_n)\to(a,c)$. Hence
\[
v_{0,n}(z)
:=
z^m\beta_p(z)
\left(
a_n+\frac{c_n}{1-\bar pz}
\right)
\]
is a non-resonant finite-gap potential converging to $u_0$ in $H^\infty_+(\T)$. By Theorem \ref{thm:gwp_finite_gap}, its corresponding solution $v_n$ is global and satisfies
\[
\sup_{t\in\R}\|v_n(t)\|_{H^s}<+\infty
\]
for every $s>0$.

Set
\[
u_{0,n}(z):=\sqrt{k}\,v_{0,n}(z^k),
\qquad
u_n(t,z):=\sqrt{k}\,v_n(k^2t,z^k).
\]
Then $u_{0,n}\to u_0^{(k)}$ in $H^\infty_+(\T)$, and Lemma \ref{lem:covering} shows that $u_n$ is a global solution. These covered data are again finite-gap potentials. Indeed, if $q_1,\ldots,q_k$ are the roots of $q^k=p$, then
\[
u_{0,n}(z)
=
z^{km}\prod_{j=1}^k\beta_{q_j}(z)
\left(
\sqrt{k}\,a_n+
\sum_{j=1}^k
\frac{c_n/\sqrt{k}}{1-\bar q_jz}
\right),
\]
and, for every fixed $j$,
\[
\sqrt{k}\overline{a_n}\frac{c_n}{\sqrt{k}}
+
\sum_{\ell=1}^k
\frac{|c_n|^2/k}{{1-\bar q_j q_\ell}}
=
\bar a_nc_n+\frac{|c_n|^2}{1-|p|^2}
=
2.
\]
Finally, Lemma \ref{lem:covering} yields
\[
\sup_{t\in\R}\|u_n(t)\|_{H^s}
\leq
k^{s+1/2}
\sup_{\tau\in\R}\|v_n(\tau)\|_{H^s}
<+\infty.
\]
This proves Corollary \ref{coro:instability}. \hfill $\qed$
}

\begin{appendix}

\section{Two useful symmetries of \eqref{eq:CS}}

\label{sec:app}

{{We discuss two symmetry transformations for solutions of \eqref{eq:CS} that will be useful in our analysis. Throughout this section, we consider smooth solutions of \eqref{eq:CS}. We begin with a discrete Galilean transformation on the torus. At this point, we do not assume that the solution $u$ is chiral; in particular, we do not impose $u(t)\in H^\infty_+(\T)$ for $t\in I$.}}

\begin{lemma}[Discrete Galilean-type symmetry] \label{lem:Galilean}
Let $I \subset \R$ be an interval with $0 \in I$ and suppose that
$$
u \in C(I; H^\infty(\T))
$$
solves \eqref{eq:CS} with initial datum $u(0) = u_0 \in H^\infty(\T)$. For every integer $m \in \Z$, define
$$
\Gc_m u(t, x) := \eu^{-i m^2 t} \eu^{imx} u(t, x-2mt) \quad \mbox{for} \quad (t,x) \in I \times \T.
$$
Then $\Gc_m u \in C(I; H^\infty(\T))$ solves \eqref{eq:CS} with initial datum $(\Gc_m u)(0,x) = \eu^{imx} u_0(x)$ for $x \in \T$. Moreover, we have the following properties.
\begin{enumerate}
\item[(i)] For every $s \geq 0$, it holds
$$
C_{m,s}^{-1} \| u(t) \|_{H^s} \leq \| \Gc_m u(t) \|_{H^s} \leq C_{m,s} \| u(t) \|_{H^s} \quad \mbox{for} \quad t \in I,
$$
with some constant $C_{m,s} > 0$ depending only on $m$ and $s$. For $s=0$, we have the isometry 
$$
\| \Gc_m u(t)\|_{L^2} = \| u(t) \|_{L^2} \quad \mbox{for} \quad t \in I.
$$
\item[(ii)]  If $u \in C(I; H^\infty_+(\T))$ and $m \in \Z_{\geq 0}$, then $\Gc_m u \in C(I; H^\infty_+(\T))$.
\end{enumerate}
\end{lemma}

\begin{remarks*}
{\em 1) {Applying $\Gc_m$ with $m<0$ to a chiral blow-up solution produces another smooth blow-up solution. It is non-chiral whenever the initial datum has a non-zero Fourier coefficient at some frequency $0\leq n<-m$. In particular, this applies to the single-bubble solutions with $m=0$ in Theorem \ref{thm:blowup}, since their constant Fourier coefficient is non-zero.}

2) For chiral solutions $u \in C(I; H^\infty_+(\T))$ and $m \in \Z_{\geq 0}$, the symmetry transformation can be written as
$$
(\Gc_m u)(t,z) = \eu^{-i m^2 t} z^m u(t, \eu^{-2imt} z) \quad \mbox{for} \quad (t,z) \in I \times \D
$$
and the action on the initial datum $u_0 \in H^\infty_+(\T)$ amounts to applying $m$ times the right shift to $u_0$, i.e., $u_0(z) \mapsto (S^mu_0 )(z) = z^m u_0(z)$.

3) The discrete Galilean symmetry above easily extends to less regular solutions of \eqref{eq:CS}, but we omit this here.}
\end{remarks*}
\begin{proof}
The {proof follows} from direct calculations as follows. Let $m \in \Z$ be given and define $y := x-2mt$ and $\Phi(t,x) := \eu^{imx-i m^2 t}$ for $t \in I$ and $x \in \T$. Then
$$
u^{(m)}(t,x) := (\Gc_m u)(t,x) = \Phi(t,x) u(t,y).
$$
It is evident that ${u^{(m)}} \in C(I; H^\infty(\T))$ holds. To see that it solves \eqref{eq:CS}, we note that a direct computation gives
$$
i \pt_t u^{(m)} = m^2 \Phi u + i \Phi \pt_t u - 2im \Phi \pt_y u, \quad
\pt_x^2 u^{(m)} = -m^2 \Phi u + 2im \Phi \pt_y u + \Phi \pt_y^2 u.
$$
Therefore,
\be \label{eq:Gal1}
i \pt_t u^{(m)} + \pt_x^2 u^{(m)} = \Phi(t,x) (i \pt_t u + \pt_y^2 u)(t,y).
\ee
Next, we consider the nonlinear term in \eqref{eq:CS}. Let $\tau_a$ denote the translation by $a$, i.e., 
$$
(\tau_a f)(x) := f(x-a).
$$
Then
$$
|u^{(m)}(t,x)|^2 = |u(t,x-2mt)|^2 = (\tau_{2mt}|u(t,\cdot)|^2)(x).
$$
Since both $\Pi$ and $D=-i \pt_x$ commute with $\tau_a$, we have
$$
D \Pi ( |u^{(m)}(t,\cdot)|^2)(x) = D \Pi(|u(t,\cdot)|^2)(x-2mt) = D \Pi(|u(t,\cdot)|^2)(y).
$$
Thus
\be \label{eq:Gal2}
D \Pi (|u^{(m)}|^2) u^{(m)} = \Phi(t,x) [ D \Pi (|u|^2) u{]}(t,y).
\ee
Combining \eqref{eq:Gal1} and \eqref{eq:Gal2}, we get
\[
\begin{aligned}
&\left [ {i}\pt_tu^{(m)}
+\pt_x^2u^{(m)}
+2D\Pi(|u^{(m)}|^2)u^{(m)} \right ](t,x)
\\
&\qquad
=
\Phi(t,x)
\left[
i\pt_t u+\pt_y^2u+2D\Pi(|u|^2)u
\right](t,y)
=
0,
\end{aligned}
\]
since $u$ solves \eqref{eq:CS}. Hence $u^{(m)} \in C(I; H^\infty(\T))$ solves \eqref{eq:CS} with initial datum $u^{(m)}(0,x) = {\eu^{imx}} u_0(x)$ for $x \in \T$.

To prove statement (i), we note by using the Fourier series that
$$
u(t,x) = \sum_{n \in \Z} \widehat{u}_n(t) \eu^{i nx}, \quad u^{(m)}(t,x) = \sum_{n \in \Z} \eu^{-im^2 t} \eu^{-2imnt} \widehat{u}_n(t) \eu^{i(n+m) x}
$$
for $t \in I$ and $x \in \T$. Hence it follows
$$
\| u^{(m)}(t) \|_{H^s}^2 = \sum_{n \in \Z} (1+(n+m)^2)^s |\widehat{u}_n(t)|^2.
$$
It is elementary to check that $c_m^{-1} (1+n^2) \leq 1+(n+m)^2 \leq c_m (1+n^2)$ for all $n \in \Z$ with the constant $c_m = 2(1+m^2)>0$. Thus we easily conclude the proof of item (i), noticing also the evident fact that $\| u^{(m)}(t) \|_{L^2} = \| u(t) \|_{L^2}$ corresponding to the case $s=0$.

Finally, if $m \geq 0$ is nonnegative and ${u} \in C(I; H^\infty_+(\T))$, then it is obvious that $\Gc_m$ preserves chirality of solutions.
\end{proof}

Next, we record another useful symmetry transformation for \eqref{eq:CS}, which involves the ({orientation-preserving}) $k$-fold cover of the torus given by the map $x \mapsto kx$ for $x \in \T$ with a given integer $k \geq 1$, where we include the trivial case $k=1$ for notational convenience.

\begin{lemma}[Discrete covering symmetry] \label{lem:covering}
Let $u \in C(I; H^\infty(\T))$ be as in Lemma \ref{lem:Galilean}. For every integer $k \geq 1$, we define
$$
\Cc_k u(t,x) :=  \sqrt{k} u(k^2 t, k x) \quad \mbox{for} \quad (t,x) \in I_k \times \T,
$$
on the rescaled time interval $I_k := \{ t \in \R \mid k^2 t \in I \}$. Then $\Cc_k u \in C(I_k; H^\infty(\T))$ solves \eqref{eq:CS} with initial datum $\Cc_k u(0,x) = \sqrt{k} u_0(kx)$ for $x \in \T$. Furthermore, the following properties hold.

\begin{itemize}
\item[(i)] For all $s \geq 0$, we have
$$
k^{1/2} \| u(k^2 t) \|_{H^s} \leq \| \Cc_k u(t) \|_{H^s} \leq k^{s+1/2} \| u(k^2 t) \|_{H^s} \quad \mbox{for} \quad t \in I_k.
$$
In particular, it holds that  $\| \Cc_k u(0) \|_{L^2} = k^{1/2} \| u_0 \|_{L^2}$.
\item[(ii)] If $u \in C(I; H^\infty_+(\T))$, then $\Cc_k u \in C(I_k; H^\infty_+(\T))$.
\end{itemize}

\end{lemma}

\begin{remarks*}{\em 
1) This symmetry transformation for \eqref{eq:CS} on $\T$ is reminiscent of the $L^2$-scaling symmetry of \eqref{eq:CM} on $\R$. However, we note {the important fact} that $\Cc_k$ does {\em not preserve} the $L^2$-mass, but rather increases it by the integer factor $k$.

2) The use of $\Cc_k$ will enable us to construct multi-bubble blow-up solutions for \eqref{eq:CS} from the single-bubble blow-up solution given by Theorem \ref{thm:blowup}. Note also that for chiral solutions, we can write
$$
{\Cc_k u(t,z)} = \sqrt{k} u(k^2 t, z^k) \quad \mbox{for} \quad (t,z) \in I_k \times \D.
$$
}
\end{remarks*}

\begin{proof}
{The case $k=1$ is immediate. Hence let $k\geq2$ be an integer.} Write $y=kx$ and $\tau = k^2 t$ for $x \in \T$ and ${t \in I_k}$. If $F(y) = \sum_{n \in \Z} \widehat{F}_n \eu^{iny}\in H^\infty(\T)$, then
$$
F(kx) = \sum_{n \in \Z} \widehat{F}_n \eu^{in kx}.
$$
Consequently, the sign of a Fourier mode is unchanged by the covering map $x \mapsto kx$, and therefore
$$
\Pi_x(F(kx)) =(\Pi_y F)(kx).
$$
Since $D_x(F(kx)) = k(D_yF)(kx)$, the function $u_k(t,x) := \Cc_k u(t,x) = \sqrt{k} u(k^2t,kx)$ satisfies
$$
i \pt_t u_k = i k^{5/2} (\pt_\tau u)(\tau, y), \quad \pt_x^2 u_k = k^{5/2} {(\pt_y^2 u)}(\tau, y), 
$$
$$
D_x \Pi_x(|u_k|^2) u_k = k^{5/2} (D_y \Pi_y(|u|^2) u)(\tau, y).
$$
From this we easily deduce that $u_k \in C(I_k; H^\infty(\T))$ solves \eqref{eq:CS} with initial datum $u_k(0,x) = \sqrt{k} u_0(kx)$ for $x \in \T$.

To show (i), let $f(x) = \sum_{n \in \Z} \widehat{f}_n \eu^{inx} \in H^\infty(\T)$ and note that
$$
f_k(x) := \sqrt{k} f(kx) = \sqrt{k} \sum_{n \in \Z} \widehat{f}_n \eu^{i kn x}. 
$$
For $s \geq 0$, we obtain
$$
\| f_k \|_{\dot{H}^s}^2 = k^{1+2s} \sum_{n \in \Z}  {|n|^{2s}} |\widehat{f}_n|^2 = k^{1+2s} \| f \|_{\dot{H}^s}^2.
$$
Hence $\| f_k \|_{L^2}^2 = k {\| f \|_{L^2}^2}$ and we readily deduce the bounds in (i).

Finally, if $u \in C(I; H^\infty_+(\T))$ is a chiral solution, we see that $u_k \in C(I_k; H^\infty_+(\T))$ as well, since $f_k(x)=\sqrt{k} f(kx) \in H^\infty_+(\T)$ for $f \in H^\infty_+(\T)$ and any integer $k \geq 2$. This proves (ii).
\end{proof}

{Remarkably,} the covering transformation preserves the finite-gap property when applied to the finite-gap potentials in Theorem \ref{thm:blowup}. Thus let us consider
\be \label{def:u_0_app}
u_0(z) = A_p z^m \beta_p(z) \left ( 1- \frac{2}{1- \bar{p}z} \right ) \quad \mbox{with} \quad \beta_p(z) = \frac{z-p}{1- \bar{p} z} 
\ee
with some $m \in \Z_{\geq 0}$, $p \in \D \setminus \{ 0 \}$ and the factor
$$
A_p = \eu^{i \phi} \sqrt{\frac{1-|p|^2}{1+|p|^2} }.
$$
We have the following result.
\begin{lemma} \label{lem:covering_finite_gap}
Let $u_0 \in H^\infty_+(\T)$ be given by \eqref{def:u_0_app}. For any integer $k \geq 2$, we define
$$
u_0^{(k)}(z) := k^{1/2} u_0(z^k) = k^{1/2} A_p z^{km} \frac{z^k-p}{1- \bar{p} z^k} \left ( 1- \frac{2}{1-\bar{p} z^k} \right ).
$$
Then $u_0^{(k)} \in H^\infty_+(\T)$ is a finite-gap potential for \eqref{eq:CS} given by
$$
{u_0^{(k)}(z) = z^{km} \prod_{j=1}^{k} \left ( \frac{z-q_j}{1- \bar{q}_j z} \right ) \cdot \left ( a + \sum_{j=1}^k \frac{c_j}{1- \bar{q}_j z} \right )},
$$
where $q_1, \ldots, q_k \in \D \setminus \{ 0 \}$ {are the roots} of $q^k = p$ and the parameters $a, c_1, \ldots{,} c_k \in \C$ satisfy the constraints
$$
{\bar{a} c_j + \sum_{\ell=1}^k \frac{c_j \bar{c}_\ell}{1- \bar{q}_j q_\ell} = 2 \quad \mbox{for} \quad j = 1, \ldots, k.}
$$ 
\end{lemma}

\begin{proof}
This follows from direct calculations. Choose some $k$-th root $q \in \D \setminus \{ 0 \}$ of $q^k = p$ and set
$$
q_j = q \eu^{2 \pi i (j-1)/k}  \quad \mbox{with $j=1, \ldots, k$},
$$
{These points run through all roots of the equation $q^k=p$.} Note that all $q_j \in \D \setminus \{ 0 \}$ are distinct. It is easy to check that
$$
\prod_{j=1}^k \beta_{q_j}(z) = \prod_{j=1}^k \frac{z- q_j}{1-\bar{q}_jz} = \beta_p(z^k) = \frac{z^k-p}{1-\bar{p} z^k}, \quad
\sum_{j=1}^k \frac{1}{1- \bar{q}_j z} = \frac{k}{1- \bar{p} z^k}.
$$
Using these identities, we find
$$
u_0^{(k)}(z) = k^{1/2} A_p z^{km} \beta_p(z^k) \left ( 1- \frac{2}{1-\bar{p}z^k} \right ) = z^{km} \prod_{j=1}^k \beta_{q_j}(z) \left ( a + \sum_{j=1}^k \frac{c_j}{1- \bar{q}_j z } \right ),
$$
where we set
$$
a := k^{1/2} A_p, \quad c_j := -\frac{2 a}{k} = -\frac{2 A_p}{k^{1/2}}.
$$
To show that $u_0^{(k)} \in H^\infty_+(\T)$ is a finite-gap potential as in \eqref{eq:finite_gap}, it remains to check the constraints \eqref{eq:constraints} with the distinct zeros $q_1,\ldots, q_k \in \D \setminus \{ 0 \}$,  $m_0 = km$ and $m_j=2$ for $j=1, \ldots, k$. Indeed, for {each fixed $j$}, we note
$$
{\sum_{\ell=1}^k \frac{1}{1- \bar{q}_j q_\ell} = \frac{k}{1- |q|^{2k}} = \frac{k}{1-|p|^2}.}
$$
Since $a = k^{1/2} A_p$ and $c_j = -2 A_p/k^{1/2}$, we obtain
\begin{align*}
{\bar{a} c_j + \sum_{\ell=1}^k \frac{c_j \bar{c}_\ell}{1- \bar{q}_j q_\ell}} & = -2 |A_p|^2 + \frac{4 |A_p|^2}{k} \frac{k}{1-|p|^2} \\
& = |A_p|^2 \left (-2 + \frac{4}{1-|p|^2} \right ) = 2 = m_j.
\end{align*}
This completes the proof.
\end{proof}

\end{appendix}

\end{document}